\documentclass[11pt]{article}
\usepackage{amsmath,amssymb,amsfonts,amsthm}
\usepackage{mathrsfs}
\usepackage{indentfirst}
\usepackage[pdftex]{color,graphicx}
\usepackage[pdftex,bookmarks,unicode,colorlinks]{hyperref}
\usepackage{enumerate} 
\usepackage[numbers]{natbib} 

\theoremstyle{plain}
\newtheorem{theorem}{Theorem}[section]
\newtheorem{proposition}[theorem]{Proposition}
\newtheorem{corollary}[theorem]{Corollary}
\newtheorem{lemma}[theorem]{Lemma}

\theoremstyle{remark}
\newtheorem{remark}[theorem]{Remark}
\newtheorem{example}[theorem]{Example}

\theoremstyle{definition}

\newtheorem{assumption}{Assumption}

\newcommand{\E}{\mathbf E}
\renewcommand{\P}{\mathbf P}
\newcommand{\T}{\mathbb T}
\newcommand{\R}{\mathbb R}

\newcommand{\Z}{\mathbb Z}
\newcommand{\Cp}{\mathbb C}
\renewcommand{\S}{\mathbb S}
\newcommand{\Bo}{{B\setminus\{0\}}}
\newcommand{\Ro}[1]{\mathbb R^{#1}\setminus \{0\}}
\renewcommand{\C}{\mathcal C}
\newcommand{\D}{\mathcal D}
\newcommand{\A}{\mathcal A}
\renewcommand{\L}{\mathcal L}
\newcommand{\e}{\epsilon}
\newcommand{\F}{\mathcal F}
\newcommand{\B}{\mathcal B}

\newcommand{\ind}{\mathbf 1}
\newcommand{\DivEps}[1]{\frac{{#1}}{\epsilon}}
\newcommand{\id}{\textbf{Id}}
\newcommand{\Lip}{\mathrm{Lip}}

\numberwithin{equation}{section} 

\begin{document}

\title{\bf\Large Homogenization of Nonlocal Partial Differential Equations Related to Stochastic Differential Equations with L\'evy Noise}
\author{\bf\normalsize{
Qiao Huang\footnote{Center for Mathematical Sciences, Huazhong University of Science and Technology, Wuhan, Hubei 430074, P.R. China. Email: \texttt{hq932309@alumni.hust.edu.cn}},
Jinqiao Duan\footnote{Department of Applied Mathematics, Illinois Institute of Technology, Chicago, IL 60616, USA. Email: \texttt{duan@iit.edu}},
Renming Song\footnote{Department of Mathematics, University of Illinois at Urbana-Champaign, Urbana, IL 61801, USA. Email: \texttt{rsong@illinois.edu}}}}

\date{}
\maketitle
\vspace{-0.3in}
\begin{abstract}
  We study the ``periodic homogenization'' for a class of nonlocal partial differential equations of parabolic-type with rapidly oscillating coefficients, related to stochastic differential equations driven by multiplicative isotropic $\alpha$-stable L\'evy noise ($1<\alpha<2$) which is nonlinear in the noise component. Our homogenization method is probabilistic. It turns out that, under suitable regularity assumptions, the limit of the solutions satisfies a nonlocal partial differential equation with constant coefficients, which are associated to a symmetric $\alpha$-stable L\'evy process. \bigskip\\
  \textbf{AMS 2010 Mathematics Subject Classification:} 60H30, 60H10, 35B27, 35R09. \\ 
  \textbf{Keywords and Phrases:} Homogenization, nonlocal parabolic PDEs, SDEs with jumps, Zvonkin's transform, strong well-posedness, Feller semigroups, Feynman-Kac formula.
\end{abstract}

\section{Introduction}

In the study of porous media, composite materials and other physical and engineering systems, one is led to the initial or boundary value problems with periodic structures (see, e.g., \cite{All02,CP99,OSY92}). The process of passing from a microscopic description to a macroscopic description of the behaviors of such systems is called \emph{homogenization}. At present, numerous publications can be found on the mathematical aspects of the homogenization theory (see \cite{BJ16,BFF17,BLM19}).

The goal of this paper is to use a \emph{probabilistic approach} to study the limit behavior, as $\epsilon\to 0$, of the solution $u^\epsilon:\R^d\to\R$ of the following \emph{nonlocal} partial differential equation (PDE) of parabolic-type with rapidly oscillating periodic and singular coefficients,
\begin{equation}\label{ue}
  \begin{cases}
     \frac{\partial u^\epsilon}{\partial t}(t,x) = \L_\e^\alpha u^\epsilon(t,x)+\left( \frac{1}{\epsilon^{\alpha-1}} e\left(\frac{x}{\epsilon}\right)+ g\left(\frac{x}{\epsilon}\right) \right)u^\e(t,x), & t>0,x\in\R^d, \\
     u^\epsilon(0,x) = u_0(x), & x\in\R^d,
  \end{cases}
\end{equation}
where $1<\alpha<2$ and the linear operator $\L_\e^\alpha$ is a nonlocal integro-differential operator of L\'evy-type given by
\begin{equation*}
  \begin{split}
     \L_\e^\alpha f(x) :=&\ \int_{\Ro d} \textstyle{\left[ f\left( x+\sigma\left(\frac{x}{\e},y\right)\right)-f(x)- \sigma^i\left(\frac{x}{\e},y\right) \partial_i f(x) \ind_{B}(y) \right] } \nu^\alpha(dy) \\
       & + \textstyle{ \left[\frac{1}{\e^{\alpha-1}} b^i\left(\frac{x}{\e}\right) + c^i\left(\frac{x}{\e}\right) \right]} \partial_i f(x),\quad x\in\R^d.
  \end{split}
\end{equation*}
Here $B$ is the unit open ball in $\R^d$ centering at the origin, and $\nu^\alpha(dy):=\frac{dy}{|y|^{d+\alpha}}$ is the isotropy $\alpha$-stable L\'evy measure.
In this paper, we use Einstein's convention that the repeated indices in a product will be summed automatically.

For notational simplicity, we introduce the linear operator $\A^{\sigma,\alpha}$ defined by
\begin{equation}\label{A}
  \A^{\sigma,\nu^\alpha} f(x) := \int_{\Ro d} \big[ f( x+\sigma(x,y))-f(x)  - \sigma^i(x,y) \partial_i f(x)\ind_B(y) \big] \nu^\alpha(dy),\quad x\in\R^d.
\end{equation}
For a function $f$ on $\R^d$ (or $F$ on $\R^d\times\R^d$), we denote $f_\epsilon(x):=f\left(\frac{x}{\e}\right)$ (or $F_\epsilon(x,y):=F\left(\frac{x}{\e},y\right)$). Then
\begin{equation}\label{rewrite_Le}
  \L^\alpha_\e=\A^{\sigma_\e,\nu^\alpha}+ \textstyle{ \left( \frac{1}{\e^{\alpha-1}} b_\e +c_\e\right) } \cdot\nabla.
\end{equation}

The main result of this paper is the following theorem. Assumptions \ref{coef1}--\ref{center} are made for the coefficients and will be listed in the sequel. We will prove it at the end of Section \ref{Homogenization}.

\begin{theorem}\label{HomoPDE}
  Under Assumptions \ref{coef1}--\ref{center}, the nonlocal PDE \eqref{ue} has a unique mild solution $u^\e$ for each $\e>0$. Moreover, for each $t\ge 0,x\in\R^d$,
  \begin{equation}\label{u_conv}
    u^\epsilon(t,x)\to u(t,x), \quad\epsilon\to 0,
  \end{equation}
  where $u$ satisfies the limit nonlocal PDE,
  \begin{equation*}
    \begin{cases}
      \frac{\partial u}{\partial t}(t,x) = \A^{\id_y,\bar\nu}u(t,x)+\bar C\cdot\nabla u(t,x)+\bar Eu(t,x), & t>0,x\in\R^d, \\
      u(0,x) = u_0(x), & x\in\R^d,
    \end{cases}
  \end{equation*}
  where
  \begin{equation*}
    \A^{\id_y,\bar\nu} f(x) := \int_{\Ro d} \left[ f( x+y)-f(x)- y^i \partial_i f(x)\ind_B(y) \right] \bar\nu(dy),
  \end{equation*}
  and the constant coefficients $\bar C,\bar E$ and measure $\bar\nu$ are given by \eqref{HomoCoef1}, \eqref{HomoCoef2} and \eqref{homo_nu} respectively. The solution is given by
  \begin{equation*}
    u(t,x)=\E[u_0(x+\bar Ct+\bar L_t)]e^{\bar Et},
  \end{equation*}
  where $\{\bar L_t\}_{t\ge0}$ is a symmetric $\alpha$-stable L\'evy processes with jump intensity measure $\bar\nu$.
\end{theorem}


The original probabilistic approach to the homogenization of \emph{local} linear second order parabolic partial differential operators is presented in \cite[Chapter 3]{BLP78}, which is based on the ergodic theorem, the Feynman-Kac formula and the functional central limit theorem. By now, there are lots of literature concerning the homogenization of second order local PDEs, i.e., the case of replacing the operator $\A^{\sigma_\e,\nu^\alpha}$ in \eqref{rewrite_Le} by a second order partial differential operator with singular coefficients. Two different scales of spatial variables involved in the coefficients have been considered in \cite{BP07}, by using the nonlinear Feynman-Kac formula in the context of backward stochastic differential equations (SDEs). In \cite{PP12}, the authors allowed the singular coefficients to be time-dependent and rapidly oscillating in time with a different scale in contrast to the spatial variable. The paper \cite{HP08} dealt with the case when the second order coefficient matrix can be degenerate, using the existence of a spectral gap and Malliavin's calculus.

There are also some literature for the homogenization of nonlocal PDEs or SDEs with jumps involved. Firstly, let us discuss several articles studying the periodic homogenization of some kinds of nonlocal operators in the \emph{analytical} aspects. It was shown in \cite{PZ17} that a homogenization model for a class of nonlocal operators with convolution-type kernels can generate a diffusion in the limit. The jump kernels are forced to have finite second moment there, which excludes the stable L\'evy densities. Later on, \cite{KPZ19} studied the homogenization of L\'evy-type  operators (or stable-like operators called by other authors) with both symmetric and non-symmetric kernels, the limit is a L\'evy-operator. Besides, an approach based on the viscosity solution of PDEs can be found in \cite{Ari09}, the author focused on the L\'evy-operator perturbed linearly by a zeroth-order term. This method has also been extended to several classes of nonlinear problems; see, e.g., \cite{Sch10}. It is worth to mention that all the homogenization problems considered in \cite{KPZ19}, \cite{Ari09} and \cite{Sch10} can cover the L\'evy-type operators (or L\'evy-operators) with full range of $\alpha\in(0,2)$, while none of them involves the drift (gradient) parts.

The \emph{probabilistic} study of homogenization of periodic stable-like processes in pure jump or jump-diffusion case can be found in \cite{Fra06,San16}. The homogenization in random medium is slightly different from the periodic case, referring to \cite{RS09} for related results for jump-diffusion processes in random medium.

The homogenization of a kind of one-dimensional pure jump Markov processes with the following form of generators has been investigated in the paper \cite{HIT77},
\begin{equation*}
  \A_\e f(x) = \int_{\R\setminus\{0\}} \left[ f(x+z)-f(x)-zf'(x) \right] \textstyle{ a\left(\DivEps x,\DivEps z\right) \frac{1}{|z|^{1+\alpha}} dz + \frac{1}{\e^{\alpha-1}} b\left(\frac{x}{\e}\right) } f'(x).
\end{equation*}
See \cite{Tom92} for a generalization for multi-dimensional case and with diffusion terms involved. In their context, the jump function $h$ is oscillating both in the spatial variable $x$ and in the noise variable $y$, and they are in the same scale. This means the noise comes from the underlying space of periodic medium. But when we do homogenization for systems with fluctuations, the noise usually comes from the external environment, so that the jump function is no longer oscillating in the noise variable, and there happens to be two different scales. As we will see in Section \ref{Pre&Ass}, the change of variables allows us to write
\begin{equation*}
  \begin{split}
     \L_\e^\alpha f(x) :=&\ \int_{\Ro d} \left[ f( x+z)-f(x)- z^i \partial_i f(x) \ind_{B}(y) \right] \textstyle{ h\left(\DivEps x,z\right)\frac{1}{|z|^{d+\alpha}} } dz \\
       & + \textstyle{ \left[\frac{1}{\e^{\alpha-1}} b^i\left(\frac{x}{\e}\right) + c^i\left(\frac{x}{\e}\right) \right] } \partial_i f(x),
  \end{split}
\end{equation*}
for some function $h$. Note that the main difference is that we do not involve oscillations for $h$ in its noise variable, while \cite{HIT77,Tom92} involve. Meanwhile, the coefficients of drift and the zeroth order term $b,c,e,g$ has two different scales.

In paper \cite{Fra07}, the author considered the homogenization of SDEs driven by multiplicative stable processes, where the noise intensity coefficient $\sigma$ is linear in the noise variable in the sense that $\sigma(x,y)=\sigma_0(x)y$, with $\sigma_0$ three-times continuously differentiable. In the  present paper, we generalize his results to the general multiplicative case. That is, the intensity function $\sigma$ need \emph{not to be linear} for the noise component. This is also more realistic in applications (see, e.g., \cite{SY04,BHR15}). For some typical forms of $\sigma$ that are nonlinear in $y$, see Example \ref{exmp}. In addition, the coefficients only need to possess some H\"older or Lipschitz continuity in our context (see Remark \ref{sigma-reg} for the comparison of the regularity assumptions for $\sigma$). This will give rise to several difficulties both in analytic and probabilistic aspects. We also use the homogenization results of SDEs to study the homogenization of the nonlocal PDEs with singular coefficients involved, by utilizing Feynman-Kac formula. The trick we use to remove the singular drift in \eqref{ue} is now known as \emph{Zvonkin's transform}, which appeared originally in \cite{Zvo74}.

The work in this paper is highly motivated by these two considerations. That is, the noise in our homogenization problem comes from  the external environment instead of the underlying periodic medium, and is not necessarily linear. Both are more suitable from the practical point of view than in earlier papers. Under these considerations, we further weaken the regularity assumptions for all coefficients in a compatible way.

Last but not least, the present paper will only focus on the \emph{subcritical} case $1<\alpha<2$. There are both probabilistic and analytical difficulties for the \emph{supercritical} case $0 < \alpha \le 1$. In analytical aspect, the well-posedness of the nonlocal Poisson equations (Section \ref{NonlocalPoisson}) in the supercritical case is still open in general setting. Intuitively, the nonlocal part has lower order than the drift part in this case, so that one cannot regard the drift as a perturbation of the nonlocal operator (in the sense of $C_0$-semigroups, as ideally the perturbing operator needs to be ``bounded'' with respect to the principal operator, cf. \cite[Secton III.2]{EN00}). This will give rise to some analytical difficulties, of which the most crucial one is the heat kernel estimates for the coupling of supercritical nonlocal operators with gradients (cf. \cite{CZ18} for the subcritical case). Although there are a few published results for these topics, they all focus on some special case such as the drifted critical ($\alpha=1$) fractional Laplacian (\cite{XZ14}). Moreover, the condition $\alpha>1$ is also crucial when doing the probabilistic homogenization, since we will have used this condition explicitly in the proofs of Lemma \ref{conv_mu} and Lemma \ref{FCLT}.

We denote by $\C^k$ ($\C_b^k$) with integer $k\ge0$ the space of (bounded) continuous functions possessing (bounded) derivatives of orders not greater than $k$.
We shall explicitly write out the domain if necessary. Denote by $\C_b(\R^d):=\C_b^0(\R^d)$, it is a Banach space with the supremum norm $\|f\|_0=\sup_{x\in\R^d}|f(x)|$. The space $\C_b^k(\R^d)$ is a Banach space endowed with the norm $\|f\|_k=\|f\|_0+\sum_{j=1}^k \|\nabla^{\otimes j}f\|$. We also denote by $\C^{\Lip}$ the class of all Lipschitz continuous functions. For a non-integer $\gamma>0$, the H\"older spaces $\C^\gamma$ ($\C^\gamma_b$) are defined as the subspaces of $\C^{\lfloor\gamma\rfloor}$ ($\C^{\lfloor\gamma\rfloor}_b$) consisting of functions whose $\lfloor\gamma\rfloor$-th order partial derivatives are locally H\"older continuous (uniformly H\"older continuous) with exponent $\gamma-\lfloor\gamma\rfloor$. These two spaces $\C^\gamma$ and $\C^\gamma_b$ obviously coincide when the underlying domain is compact. The space $\C^\gamma_b(\R^d)$ is a Banach space endowed with the norm $\|f\|_\gamma=\|f\|_{\lfloor\gamma\rfloor}+[\nabla^{\lfloor\gamma\rfloor}f]_{\gamma-\lfloor\gamma\rfloor}$, where the seminorm $[\cdot]_{\gamma'}$ with $0<\gamma'<1$ is defined as $[f]_{\gamma'}:=\sup_{x,y\in\R^d,x\ne y}\frac{|f(x)-f(y)|}{|x-y|^{\gamma'}}$ (this seminorm can also be defined for the case $\gamma'=1$, which is exactly the Lipschitz seminorm). In the sequel, the torus $\T^d:=\R^d/\Z^d$ will be used frequently. Denote by $\D:=\D(\R_+;\T^d)$ the space of all $\T^d$-valued c\`adl\`ag functions on $\R_+$, equipped with the Skorokhod topology. We shall always identify the periodic function on $\R^d$ of period 1 with its restriction on the torus $\T^d=\R^d/\Z^d$. This allows us to regard the space $\C^k(\T^d)$ ($\C^\gamma(\T^d)$) as a sub-Banach space of $\C_b^k(\R^d)$ ($\C_b^\gamma(\R^d)$).

By $B_r$ we means the open ball in $\R^d$ centering at the origin with radius $r>0$, we shall omit the subscript when the radius is one. The capital letter $C$ denotes a finite positive constant whose value may vary from line to line. We also use the notation $C(\cdots)$ to emphasize the dependence on the quantities appearing in the parentheses.

The remainder of the paper is organized as follows. In Section \ref{Pre&Ass}, we present some general assumptions and preliminary results. In Section \ref{NonlocalPoisson}, we study the well-posedness of the nonlocal Poisson equation and the Feller properties of the semigroup associated with generator of the form \eqref{rewrite_Le}. Section \ref{SDEs} is devoted to the strong well-posedness and exponential ergodicity of the L\'evy driven SDE. As a consequence, we obtained the Feynman-Kac representation for the nonlocal PDE \eqref{ue}. Lastly, Section \ref{Homogenization} contains the homogenization results of SDEs and nonlocal PDEs, utilizing the ergodicity and the Feynman-Kac representation.

\section{Preliminaries and general assumptions}\label{Pre&Ass}

Let $(\Omega,\F,\P,\{\F_t\}_{t\ge0})$ be a filtered probability space endowed with a Poisson random measure $N^\alpha$ on $(\Ro d)\times\R_+$ with jump intensity measure $\nu^\alpha(dy)=\frac{dy}{|y|^{d+\alpha}}$, where $1<\alpha<2$. Denote by $\tilde N$ the associated
compensated Poisson random measure, that is, $\tilde N^\alpha(dy,ds):=N^\alpha(dy,ds)-\nu^\alpha(dy)ds$. We assume that the filtration $\{\F_t\}_{t\ge0}$ satisfies the usual conditions. Let $L^\alpha=\{L_t^\alpha\}_{t\ge 0}$ be a $d$-dimensional isotropic $\alpha$-stable L\'evy process given by
\begin{equation*}
  L_t^\alpha=\int_0^t\int_{\Bo}y\tilde N^\alpha(dy,ds)+ \int_0^t\int_{B^c}y N^\alpha(dy,ds).
\end{equation*}

Given $\epsilon> 0, x\in\R^d$, consider the following:
\begin{equation}\label{Xe}
  \textstyle{ dX_t^{x,\epsilon} = \left(\frac{1}{\epsilon^{\alpha-1}} b\left(\frac{X_{t}^{x,\epsilon}}{\epsilon}\right)+ c\left(\frac{X_{t}^{x,\epsilon}}{\epsilon}\right)\right)dt + \sigma\left(\frac{X_{t-}^{x,\epsilon}}{\epsilon},dL_t^{\alpha}\right) }, \quad X_0^{x,\epsilon}=x,
\end{equation}
or more precisely,
\begin{equation*}
  \begin{split}
     X_t^{x,\epsilon} = &\ x+\int_0^t \textstyle{ \left(\frac{1}{\epsilon^{\alpha-1}} b\left(\frac{X_{s}^{x,\epsilon}}{\epsilon}\right) + c\left(\frac{X_{s}^{x,\epsilon}}{\epsilon}\right)\right) } ds \\
       &\ +\int_0^t\int_{\Bo} {\textstyle{ \sigma\left( \frac{X_{s-}^{x,\epsilon}}{\epsilon},y\right) }} \tilde N^\alpha(dy,ds)+ \int_0^t\int_{B^c} {\textstyle{ \sigma \left(\frac{X_{s-}^{x,\epsilon}}{\epsilon},y\right) }} N^\alpha(dy,ds),
  \end{split}
\end{equation*}
where the coefficients $b,c,\sigma(\cdot,y)$ are periodic, for each $y\in\R^d$, of periodic one in each component. The shorthand notation for the stochastic differential term in \eqref{Xe} is due to \cite{Mas07}.

Define $\tilde X_t^{x,\epsilon} :=\frac{1}{\epsilon}X^{x,\epsilon} _{\epsilon^\alpha t}$. It is easy to check that
\begin{equation}\label{Xe_tilde}
  d\tilde X_t^{x,\epsilon} = \left( b(\tilde X_{t}^{x,\epsilon})+ \epsilon^{\alpha-1} c(\tilde X_{t}^{x,\epsilon})\right)dt+ \textstyle{\DivEps 1}
  \sigma\left(\tilde X_{t-}^{x,\epsilon},\e d\tilde L_t^{\alpha}\right), \quad \tilde X_0^{x,\epsilon}=\frac{x}{\epsilon},
\end{equation}
where $\{\tilde L_t^\alpha\}:=\{\frac{1}{\epsilon}L^\alpha_{\epsilon^\alpha t}\}\stackrel{\mathtt d}{=}\{L_t^\alpha\}$ by virtue of the selfsimilarity. We shall also consider the ``limit'' equation, namely
\begin{equation}\label{X_tilde}
  d\tilde X_t^x = b(\tilde X^x_{t})dt+
  \bar\sigma \left(\tilde X^x_{t-},d\tilde L_t^{\alpha}\right), \quad \tilde X_0^x=x,
\end{equation}
where, intuitively, $\bar\sigma$ is the limit of $\frac{1}{\e}\sigma(\cdot,\e\cdot)$. We will make the relation between $\sigma$ and $\bar\sigma$ clear in the forthcoming assumption \ref{scaling}. For notational simplicity, we shall allow the parameter $\e$ to be zero in $\tilde X^{x,\e}$ to include $\tilde X^x$, i.e., $\tilde X^{x,0}:=\tilde X^x$.

In the sequel, we will regard the solutions $\tilde X^{x,\e},\tilde X^x$ of \eqref{Xe_tilde} and \eqref{X_tilde} as $\T^d$-valued processes, by mapping all trajectories of the processes on $\R^d$ to the torus $\T^d$, via the canonical quotient map $\pi:\R^d\to\R^d/\Z^d$. Then the periodicity of the coefficients implies that $\tilde X^{x,\e}$ and $\tilde X$ are well-defined stochastic processes on $\T^d$ (cf. \cite[Section 3.3.2]{BLP78}).

Now we list some general assumptions for the nonlocal PDE \eqref{ue} and the SDE \eqref{Xe}. All these assumptions are assumed to hold in the sequel unless otherwise specified.
\begin{assumption}\label{coef1}
  The functions $b,c,e,g,u_0$ are all periodic of period 1 in each component. For every $y\in\R^d$, the function $x\to\sigma(x,y)$ is periodic of period 1 in each component.
\end{assumption}

\begin{assumption}\label{regular}
  The functions $b,c,e$ are of class $\C_b^\beta$ with exponent $\beta$ satisfying
  \begin{equation*}
    1-\frac{\alpha}{2}<\beta<1.
  \end{equation*}
  The functions $g$ and $u_0$ are both continuous.
\end{assumption}


\begin{assumption}\label{sigma}
  The function $\sigma:\R^d\times\R^d\to\R^d$ satisfies the following conditions.

  (1). \emph{Regularity}. For every $x\in\R^d$, the function $y\to\sigma(x,y)$ is of class $\C^2$. There exists $\alpha-1<\lambda\le1$ such that
  \begin{equation}\label{sigma-Holder}
    \sup_{x\in\R^d}[\nabla_y \sigma\left(x, \cdot\right)]_\lambda < \infty.
  \end{equation}
  There exists a constant $C>0$, such that for any $x_1,x_2,y\in\R^d$,
  \begin{equation*}
    |\sigma(x_1,y)-\sigma(x_2,y)|\le C|x_1-x_2| |y|.
  \end{equation*}

  (2). \emph{Oddness}. For all $x,y\in\R^d$, $\sigma(x,-y)=-\sigma(x,y)$.


  (3). \emph{Bounded inverse Jacobian}. The Jacobian matrix with respect to the second variable $\nabla_y\sigma(x,y)$ is non-degenerate for all $x,y\in\R^d$, and there exists a constant $C>0$ such that $|(\nabla_y\sigma(x,y))^{-1}| \le C$ for all $x,y\in\R^d$, where $|\cdot|$ is the operator norm on $\mathscr L(\R^d,\R^d)$.

  (4). \emph{Growth condition}. There exists a positive bounded measurable function $\phi:\R^d\to\R_+$, such that for all $x,y\in\R^d$,
  \begin{equation*}
    \phi(x)^{-1}|y| \le |\sigma(x,y)| \le \phi(x)|y|.
  \end{equation*}



\end{assumption}


\begin{remark}\label{rem_sigma}
  Some comments on our assumptions will be helpful:

  (1). As mentioned in the end of the introduction, $b,c,e,g,u_0$ and the function $x\to\sigma(x,y)$, for every $y\in\R^d$, can be regarded as functions on $\T^d$. Hence we have $b,c,e\in\C^\beta(\T^d)$, $g,u_0\in\C(\T^d)$, under Assumptions \ref{coef1} and \ref{regular}. The condition \eqref{sigma-Holder} reduce to
  \begin{equation*}
    \sup_{x\in\T^d}[\nabla_y \sigma\left(x, \cdot\right)]_\lambda < \infty.
  \end{equation*}

  (2). In Assumption \ref{sigma}, the condition \eqref{sigma-Holder} should not be confused with the condition that $y\to\sigma(x,y)$ is of class $\C^2$. The former requires the uniform H\"older continuity of $\sigma$ in $y$, as indicated in the definition of the H\"older seminorm $[\cdot]_\lambda$, while the latter do not force the uniformity. Therefore, none of them can imply the other.

  (3). Both the oddness and the growth condition in Assumption \ref{sigma} imply that $\sigma(\cdot,0)\equiv0$.

  (4). The bounded inverse Jacobian condition implies that $|\nabla_y\sigma| \ge C^{-1}$. Since by Hadamard's inequality (see, for instance, \cite{Sch70}),
  \begin{equation}\label{bij}
    |(\nabla_y\sigma)^{-1}| \le C \Rightarrow |\det((\nabla_y\sigma)^{-1})| \le C^d \Leftrightarrow |\det(\nabla_y\sigma)| \ge C^{-d} \Rightarrow |\nabla_y\sigma| \ge C^{-1}.
  \end{equation}

  (5). The growth condition implies that for any $\gamma>\alpha$, we have
  \begin{equation}\label{moment}
    \sup_{x\in\R^d}\int_{\Bo} |\sigma(x,y)|^{\gamma} \nu^\alpha(dy)<\infty.
  \end{equation}
  This ensures that we can apply It\^o's formula to $f(\tilde X^x_t)$ (or $f(\tilde X^{x,\e}_t)$, $f(X^{x,\e}_t)$), for any $f\in\C_b^{\gamma}(\R^d)$ with $\gamma>\alpha$ (cf. \cite[Lemma 4.2]{Pri12}).

  (6). By virtue of the oddness condition in Assumption \ref{sigma} and the symmetry of the jump intensity measure $\nu^\alpha$, for any $x \in \R^d$,
  \begin{equation}\label{trans}
    \text{P.V.}\int_{\sigma(x,\cdot)^{-1}B\setminus B}\sigma^i(x,y)\nu^\alpha(dy)=\text{P.V.}\int_{B\setminus\sigma(x,\cdot)^{-1} B}\sigma^i(x,y)\nu^\alpha(dy)=0.
  \end{equation}
  Consequently we can rewrite the operator $\A^{\sigma,\nu^\alpha}$ in \eqref{A} as
  \begin{equation}\label{A_rewrite}
    \A^{\sigma,\nu^\alpha} f(x) = \int_{\Ro d} \left[ f( x+z)-f(x)- z^i \partial_i f(x)\ind_B(z) \right] \nu^{\sigma,\alpha}(x,dz),
  \end{equation}
  where the kernel $\{\nu^{\sigma,\alpha}(x,\cdot)|x\in\R^d\}$ is given by
  \begin{equation}\label{kernel}
    \nu^{\sigma,\alpha}(x,A):=\int_{\Ro d}\ind_A(\sigma(x,y))\nu^\alpha(dy), \quad A\in\B(\Ro d).
  \end{equation}
  Moreover, for any $\gamma>\alpha$, the growth condition in Assumption \ref{sigma} implies that
  \begin{equation}\label{Levy-kernel}
    \begin{split}
       &\ \sup_{x\in\R^d}\int_{\Ro d} (|z|^{\gamma}\wedge1) \nu^{\sigma,\alpha}(x,dz) = \sup_{x\in\R^d}\int_{\Ro d} (|\sigma(x,y)|^{\gamma}\wedge1) \nu^\alpha(dy) \\
         \le&\ \sup_{x\in\R^d}\left(\int_{|y|\le\phi(x)} (\phi(x)|y|)^{\gamma} \nu^\alpha(dy) + \int_{|y|\ge\phi(x)^{-1}}\nu^\alpha(dy) \right) \\
         \le&\ \frac{1}{\gamma-\alpha} \|\phi\|_{L^\infty}^{2\gamma-\alpha} + \frac{1}{\alpha} \|\phi\|_{L^\infty}^{\alpha} < \infty.
    \end{split}
  \end{equation}
\end{remark}

\begin{remark}\label{sigma-reg}
  The special case $\sigma(x,y)=\sigma_0(x)y$ with certain $\sigma_0$ is considered in \cite{Fra07}, where the author assumed the function $\sigma_0:\R^d\to \text{GL}(\R^d)$ is periodic and of class $\C^3$. In our context, Assumption \ref{coef1} and \ref{sigma} amounts to saying that
  \begin{equation}\label{sigma_0}
    \sigma_0:\R^d\to \text{GL}(\R^d) \text{ is periodic and Lipschitz}. 
  \end{equation}
  Since these imply the regularity condition immediately,
  the bounded inverse Jacobian and growth conditions are fulfilled by continuity and periodicity, together with the observation $\sup_{x\in\R^d} \|\sigma_0(x)\| \vee \|\sigma_0(x)^{-1}\| < \infty$. The oddness condition is trivial in this case.
\end{remark}

In practice, the noise is not always linear. Here we give some nontrivial examples for $\sigma$, that is, nonlinear in $y$.
\begin{example}\label{exmp}
  Suppose $\sigma_0$ to satisfy \eqref{sigma_0}.

  (i). The dependence of the noise is a small perturbation of the linear case, namely, $\sigma(x,y)=\sigma_0(x)y+\delta(x,y)$, where the function $\delta$ satisfies the same properties as $\sigma$ in Assumptions \ref{coef1} and (i), (ii) in \ref{sigma}, but has much smaller scale than $\sigma_0(x)y$, so that Assumptions \ref{coef1}-\ref{sigma} and the forthcoming Assumption \ref{Jacobian} do hold for $\sigma$.

  (ii). Another case is that the function $\sigma$ is separable but not linear in $y$. To be precise, let $\eta:\R^d\to\R^d$ be an odd function of class $\C^2$, satisfying that $\nabla\eta$ is \emph{uniformly} H\"older continuous with exponent $\lambda>\alpha-1$ and $\nabla\eta(y)$ is non-degenerate for all $y\in\R^d$, and there exist some constants $C_1,C_2>0$, such that $|\nabla\eta|\ge C_1$ and $C_2^{-1} |y|\le|\eta(y)|\le C_2 |y|$. Now let $\sigma(x,y)=\sigma_0(x)\eta(y)$. Then $\sigma$ satisfies Assumption \ref{sigma}.

  (iii). Combining the above two example together, one can obtain a more general example, that is, $\sigma(x,y) = \sigma_0(x)\eta(y)+\delta(x,y)$.
\end{example}

We will need some regularities for the 'partial' inverse of $\sigma$. For a function $F:\R^d\times\R^d\ni(x,y)\to F(x,y)\in\R$, we say $F\in L^\infty_2(\R^d;\C^{\Lip}_1(\R^d;\R^d))$, if there exists a constant $C>0$ such that for all $x,y \in\R^d$, $|F(x,y)|\le C$, and for all $x_1,x_2,y \in\R^d$, $|F(x_1,y)-F(x_2,y)|\le C|x_1-x_2|$. Then the regularity and growth conditions in Assumption \ref{sigma} imply that the function $(x,y)\to\sigma(x,y)/|y|$ is of class $L^\infty_2(\R^d;\C^{\Lip}_1(\R^d;\R^d))$.

\begin{lemma}\label{inverse}
  Under Assumption \ref{sigma}, for every $x\in\R^d$, the function $y\to\sigma(x,y)$ is a $\C^2$-diffeomorphism. Denote the inverse by $\tau(x,z):=\sigma(x,\cdot)^{-1}(z)$, then for every $z\in\R^d$, the function $x\to\tau(x,z)$ is periodic of period one in each component. Moreover, the function $(x,z)\to\tau(x,z)/|z|$ is of class $L^\infty_2(\R^d;\C^{\Lip}_1(\R^d;\R^d))$.
\end{lemma}

\begin{proof}
  Fix $x\in\R^d$. Since the function $y\to\sigma(x,y)$ is of class $\C^2$, by the bounded inverse Jacobian condition in Assumption \ref{sigma}, together with Hadamard's global inverse function theorem (see \cite[Theorem 6.2.4]{KP12}), $\sigma(x,\cdot)$ is a $\C^2$-diffeomorphism. The periodicity is obvious.

  Now using the bounded inverse Jacobian condition, the Jacobian matrix of $\tau(x,z)$ with respect to $z$ satisfies $|\nabla_z\tau(x,z)|\le C$, for all $x,z\in\R^d$. Then by the growth condition and regularity condition, the second assertion follows from the following derivation,
  \begin{equation*}
    \sup_z\frac{|\tau(x,z)|}{|z|}\le\phi(x),
  \end{equation*}
  \begin{equation*}
    \begin{split}
       &\ \sup_z\frac{|\tau(x_1,z)-\tau(x_2,z)|}{|z|} = \sup_y\frac{|\tau(x_1,\sigma(x_1,y))- \tau(x_2,\sigma(x_1,y))|}{|\sigma(x_1,y)|} \\
         =&\ \sup_y\frac{|\tau(x_2,\sigma(x_2,y))- \tau(x_2,\sigma(x_1,y))|}{|\sigma(x_1,y)|} \le \|\phi\|_{L^\infty} \|\nabla_z\tau\|_{L^\infty} \sup_y\frac{|\sigma(x_2,y)-\sigma(x_1,y)|}{|y|} \\
         \le&\ C \|\phi\|_{L^\infty} \|\nabla_z\tau\|_{L^\infty} |x_1-x_2|.
    \end{split}
  \end{equation*}
\end{proof}

\begin{assumption}\label{Jacobian}
  $\det(\nabla_z\tau) \in L^\infty_2(\R^d;\C^{\Lip}_1(\R^d;\R))$.
\end{assumption}

This assumption is rather mild, as shown in the following remark.

\begin{example}
  (i). In the case $\sigma(x,y)=\sigma_0(x)y$, the partial inverse is $\tau(x,z) = \sigma_0(x)^{-1}y$, the Jacobian of $\tau(x,z)$ with respect to $z$ is $\nabla_z\tau(x,z)\equiv\sigma_0(x)^{-1}$. Then Assumption \ref{Jacobian} reduces to that the function $\det(\sigma_0)^{-1}:\R^d\to\R$ is Lipschitz, while this holds automatically since it is a direct consequence of \eqref{sigma_0}.

  (ii). When $\sigma(x,y) = \sigma_0(x)\eta(y)$ as in Example \ref{exmp}.(ii), $\tau(x,z) = \eta^{-1}(\sigma_0(x)^{-1}z)$. Then it is easy to deduce that Assumption \ref{Jacobian} is implied by the bounded inverse Jacobian condition in Assumption \ref{sigma}, using a similar argument as \eqref{bij}. \qed
\end{example}

If we let
\begin{equation}\label{h}
  h(x,z)=|\det \nabla_z\tau(x,z)| \frac{|z|^{d+\alpha}}{|\tau(x,z)|^{d+\alpha}},
\end{equation}
then by \eqref{kernel}, $\nu^{\sigma,\alpha}(x,dz)=h(x,z)\frac{dz}{|z|^{d+\alpha}}$. Using the growth condition, we also find that for all $x,z\in\R^d$,
\begin{equation}\label{upper_lower_bound}
  \|\phi\|_{L^\infty}^{-1}\le\frac{|\tau(x,z)|}{|z|}\le \|\phi\|_{L^\infty}.
\end{equation}
Combining \eqref{h}, \eqref{upper_lower_bound}, Lemma \ref{inverse} and Assumption \ref{Jacobian}, together with the fact that if $f,g\in\C^\gamma$ and $\inf |g|>0$, then $f/g\in\C^\gamma$, we conclude that

\begin{lemma}\label{regularity_h}
  Under Assumptions \ref{sigma} and \ref{Jacobian}, $h\in L^\infty_2(\R^d;\C^{\Lip}_1(\R^d;\R^d))$, namely, there exists a constant $h_0>0$ such that $|h(x_1,z)-h(x_2,z)|\le h_0|x_1-x_2|$ for all $x_1,x_2,z\in\R^d$. Moreover, there also exists a constant $h_1>1$ such that $h_1^{-1}\le h(x,z)\le h_1$ for all $x,z\in\R^d$.
\end{lemma}

In particular, the kernel $\nu^{\sigma,\alpha}$ is comparable to the jump intensity measure of an isotropic $\alpha$-stable process.

\begin{remark}\label{useful-results}
   Thanks to Lemma \ref{regularity_h}, the general assumptions in \cite{Bas09,CZ18} are satisfied. Thus, the regularity results and heat kernel estimates therein are available in our context. Actually, these two papers only need that $h\in L^\infty_2(\R^d;\C^\gamma_1(\R^d;\R^d))$ for some $0<\gamma<1$, this is the case by virtue of the natural embedding $\C^{\Lip}\subset\C^\gamma$. Note that \cite{Bas09} also needs $\alpha+\beta$ not to be an integer, this can be fulfilled by choosing an appropriate $\beta$.
\end{remark}


Now we clarify that the function $\bar\sigma$ appeared in the ``limit'' equation \eqref{X_tilde} is the scaling limit of $\sigma$. Firstly, we use the L'H\^opital's rule to proceed the limit,
\begin{equation*}
  \lim_{\e\to0} {\textstyle{\DivEps1}} \sigma(x,\e y) = \lim_{\e\to0}(\nabla_y\sigma(x,\e y))y = (\nabla_y\sigma(x,0))y,
\end{equation*}
by the continuity of $\nabla_y\sigma(x,y)$ with respect to $y$. If we set
\begin{align}
  \bar\sigma_0&(x) := \nabla_y\sigma(x,0), \label{sigma0-bar}\\
  \bar\sigma&(x,y) := \bar\sigma_0(x)y, \label{sigma-bar}
\end{align}
then $\bar\sigma$ is the point-wise limit of $\frac{1}{\e}\sigma(\cdot,\e\cdot)$ as $\e\to0$. However, as we will see in Section \ref{SDEs} (more precisely, Lemma \ref{conv_mu}), the point-wise convergence is not enough when we deal with the ergodicity. We need a stronger convergence as follows:

\begin{assumption}\label{scaling}
  For every $y\in\R^d$, $\DivEps1\sigma(x,\e y)\to (\nabla_y\sigma(x,0))y$ 
  uniformly in $x\in\R^d$, as $\e\to0$. 
\end{assumption}

We can say more for the scaling limit $\bar\sigma$.
%

\begin{remark}\label{sigma-conv}
  The function $\bar\sigma_0:\R^d\to \text{GL}(\R^d)$ defined in \eqref{sigma0-bar} is periodic and Lipschitz. Indeed, by a straightforward limit argument, one can easily verify that Assumptions \ref{coef1} and \ref{sigma} holds for $\bar\sigma$. This implies that $\bar\sigma_0$ has the same property as $\sigma_0$ in Remark \ref{sigma-reg}. Moreover, if we denote the partial inverse of $\bar\sigma$ with respect to the second variable by $\bar\tau$, and denote the associated function as in \eqref{h} by $\bar h$, then Lemma \ref{regularity_h} also holds for $\bar h$.
\end{remark}

\begin{example}
  (i). Obviously, in the linear case that $\sigma(x,y)=\sigma_0(x)y$, Assumption \ref{scaling} automatically holds with $\bar\sigma = \sigma$.

  (ii). For the case in Example \ref{exmp}.(ii), $\sigma(x,y)=\sigma_0(x)\eta(y)$, Assumption \ref{scaling} also holds automatically, with $\bar\sigma(x,y) = \sigma_0(x)(\nabla\eta(0))y$. Indeed, this assumption reduces to $\frac{1}{\e}\eta(\e y) \to (\nabla\eta(0))y$ for every $y$, while this is a consequence of L'H\^opital's rule.

  (iii). In the perturbation case that $\sigma(x,y)=\sigma_0(x)y+\delta(x,y)$ as in Example \ref{exmp}.(i), Assumption \ref{scaling} reduces to that for every $y\in\R^d$, $\DivEps1\delta(x,\e y)\to(\nabla_y\delta(x,0))y$ uniformly in $x\in\R^d$, as $\e\to0$.
\end{example}

Finally, we need a centering assumption on $b$ and $e$ when and only when doing the homogenization.
\begin{assumption}\label{center}
  The functions $b$ and $e$ satisfy the \emph{centering condition},
  \begin{equation*}
    \int_{\T^d}b(x)\mu(dx)=0, \qquad \int_{\T^d}e(x)\mu(dx)=0,
  \end{equation*}
  where $\mu$ is the invariant probability measure on $(\T^d,\B(\T^d))$ of the solution $\tilde X^x$ of \eqref{X_tilde}.
\end{assumption}
We remark here that the well-posedness of \eqref{X_tilde} and the existence and uniqueness of the invariant probability measure $\mu$ will be clear in Proposition \ref{strong-sol} and Lemma \ref{inv}, respectively.

Note that the centering assumption \ref{center} is quite common and natural in the homogenization problems and the reader can also find it in \cite{BLP78,HP08,Par99}.

\section{Nonlocal Poisson equation with zeroth-order term}\label{NonlocalPoisson}

As mentioned in the introduction, we will apply Zvonkin's transform to study the homogenization of SDEs and nonlocal PDEs. Before that, we shall investigate the strong well-posedness of all the SDEs presented in the previous section, and Zvonkin's transform will also play an important role in this step (see next section). For the generality, we will focus on SDE \eqref{X_tilde} with $\bar\sigma$ \emph{only} satisfying Assumption \ref{coef1}-\ref{Jacobian} in the whole of this section as well as Subsection \ref{subsec:4-1}, since SDEs \eqref{Xe} and \eqref{Xe_tilde} are basically of the same form.

The key is to consider the following nonlocal Poisson equation with zeroth-order term,
\begin{equation}\label{resovent}
  \kappa u-\L^\alpha u = f,
\end{equation}
where $\kappa>0$, and $\L^\alpha$ is the linear integro-partial differential operator given by
\begin{equation}\label{L}
  \L^\alpha:=\A^{\bar\sigma,\nu^\alpha}+b\cdot\nabla,
\end{equation}
which may be regarded as the infinitesimal generator of the solution process $\tilde X$ of \eqref{X_tilde} once we prove its well-posedness in the next section.

\subsection{Well-posedness of nonlocal Poisson equation}

We first adapt the maximum principle in \cite[Proposition 3.2]{Pri12} as follows. The proof is almost the same and the difference is provided in Appendix \ref{app}. 

\begin{lemma}\label{MP}
  If $u\in\C_b^{1+\gamma}(\R^d), 1+\gamma>\alpha$, is a solution to $\kappa u-\L^\alpha u = f$ with $\kappa>0$ and $f\in\C_b(\R^d)$, then
  \begin{equation*}
    \kappa\|u\|_0\le \|f\|_0.
  \end{equation*}
\end{lemma}

Now we investigate the solvability of the Poisson equation with a zeroth-order term involved. The results generalize the Schauder estimates in \cite{Pri12} to the anisotropic nonlocal case.

\begin{proposition}\label{res_prob}
  For any $\kappa>0$ and $f\in\C_b^\beta(\R^d)$, where $\beta$ is the exponent in Assumptions \ref{regular}, the nonlocal Poisson equation \eqref{resovent} has a unique solution $u=u_\kappa\in\C_b^{\alpha+\beta}(\R^d)$. In addition, there exists a positive constant $C=C(\kappa,\|b\|_\beta)$ such that
  \begin{equation}\label{energy1}
    \|u_\kappa\|_{\alpha+\beta} \le C (\|u_\kappa\|_0+\|f\|_\beta).
  \end{equation}
\end{proposition}

\begin{proof}
  The a priori estimate \eqref{energy1} is from \citep[Theorem 7.1, Theorem 7.2]{Bas09}. We thus need to show that the equation \eqref{resovent} has a unique solution $u_\kappa\in\C_b^{\alpha+\beta}(\R^d)$.

  Now we prove the existence and uniqueness of solution in $\C_b^{\alpha+\beta}(\R^d)$. It is shown in \cite[Theorem 3.4]{Pri12} that when $\bar\sigma(\cdot,y)\equiv\id_y(y):=y$, the existence and uniqueness hold in $\C_b^{\alpha+\beta}(\R^d)$.
  For the general $\bar\sigma$, we apply the method of continuity (see \cite[Section 5.2]{GT01}).

  Define a family of linear operators by $\L_\theta:=\theta\A^{\bar\sigma,\nu^\alpha}+ (1-\theta)\A^{\id_y,\nu^\alpha}+b\cdot\nabla$. We consider the family of equations:
  \begin{equation*}
    \kappa u-\L_\theta u=f.
  \end{equation*}
  We can also rewrite the nonlocal term in $\L_\theta$ into the form \eqref{A_rewrite}, with the kernel given by $\nu_\theta:=\theta\nu^{\bar\sigma,\alpha}+(1-\theta)\nu^\alpha$. Then the a priori estimate \eqref{energy1} also holds for $u_\theta$ (cf. Remark \ref{useful-results}). As a result, the operator $\L_\theta$ can be considered as a bounded linear operator from the Banach space $\C_b^{\alpha+\beta}(\R^d)$ into the Banach space $\C_b^{\beta}(\R^d)$.

  Note that $\L_0=\A^{\id_y,\nu^\alpha}+b\cdot\nabla$, which is the case considered in \cite{Pri12}, and $\L_1=\L^\alpha$.
  The solvability of the equation \eqref{resovent} for any $f\in\C_b^\beta(\R^d)$ is then equivalent to the invertibility of the operator $\L_\theta$. We can see from the proof of Lemma \ref{MP} that $\|u_\theta\|_0 \le C \|f\|_0$. Then together with the estimate \eqref{energy1} for $u_\theta$, we have the bound
  \begin{equation*}
    \|u_\theta\|_{\alpha+\beta} \le C \|f\|_\beta,
  \end{equation*}
  with the constant $C$ independent of $\theta$. Since, as discussed in \cite{Pri12}, the operator $\L_0=\A^{\id_y,\nu^\alpha}+b\cdot\nabla$ maps $\C_b^{\alpha+\beta}(\R^d)$ onto $\C_b^{\beta}(\R^d)$, the method of continuity is applicable and the result follows.
\end{proof}

\begin{remark}
  If we take the periodicity assumption \ref{coef1} into account, then we can slightly strengthen the conclusions in Proposition \ref{res_prob}. That is, if $f\in\C^\beta(\T^d)$, then the unique solution of \eqref{resovent} is of class $\C^{\alpha+\beta}(\T^d)$.
\end{remark}

\subsection{Feller property and heat kernels}

In this subsection, we will study further the operator $\L^\alpha$. All these results will be used in the next section. First of all, the classical theory of $C_0$-semigroups yields that $\L^\alpha$ is the generator of a Feller semigroup. The proof is standard and provided in Appendix \ref{app}.

\begin{lemma}\label{Feller}
  The linear operator $(\L^\alpha,D(\L^\alpha))$, $D(\L^\alpha)=\C^{\alpha+\beta}(\T^d)$, defined on the Banach space $(\C(\T^d),\|\cdot\|_0)$, is closable and dissipative, its closure generates a Feller semigroup $\{P_t\}_{t\ge0}$ on $\C(\T^d)$.
\end{lemma}

Let us recall the notion of martingale problem (see \cite[Section 4.3]{EK09}). First recall that $\D=\D(\R_+;\T^d)$ is the space of all $\T^d$-valued c\`adl\`ag functions on $\R_+$, equipped with the Skorokhod topology. Let $w_t(\omega)=\omega(t),\omega\in\D$, be the coordinate process on $(\D,\B(\D))$, and $\{\F^w_t\}_{t\ge0}:=\sigma(w_s:0\le s\le t)$ be the canonical filtration. Given a probability measure $\nu$ on $\T^d$, we say that a probability measure $\P^\nu$ on $(\D,\B(\D))$ is a solution of the martingale problem for $(\L^\alpha,\nu)$, if $\P^\nu\circ w_0^{-1}=\nu$ and the process
\begin{equation*}
  M^f(t):=f(w_t)-f(w_0)-\int_0^t \L^\alpha f(w_s) ds
\end{equation*}
is a $(\D,\B(\D),\{\F^w_t\}_{t\ge0},\P^\nu)$-martingale, for any $f\in D(\L^\alpha)=\C^{\alpha+\beta}(\T^d)$. We denote by $\delta_x$ the Dirac measure, or equivalently, the Dirac function as distribution, focusing on $x\in\R^d$.

The following heat kernel results can be found in references, see Appendix \ref{app} for details.
\begin{lemma}\label{martingale-prob}
  For every $x\in\T^d$, the martingale problem for $(\L^\alpha,\delta_x)$ has a unique solution $\P^x$. Moreover, the coordinate process $\{w_t\}_{t\ge0}$ is a Feller process with generator the closure of $(\L^\alpha,\C^{\alpha+\beta}(\T^d))$, and has a jointly continuous transition probability density $p(t;x,y)$, i.e., $\P^x(w_t\in A)=\int_A p(t;x,y)dy$, $A\in\B(\T^d)$, which satisfies for each $T>0$,
  \begin{gather*}
    \textstyle{ C_1^{-1}\sum_{j\in\Z^d}\left( \frac{t}{|x-y+j|^{d+\alpha}}\wedge t^{-\frac{d}{\alpha}} \right) \le p(t;x,y)\le C_1\sum_{j\in\Z^d}\left( \frac{t}{|x-y+j|^{d+\alpha}}\wedge t^{-\frac{d}{\alpha}} \right) }, \\
    \textstyle{ |\nabla_x p(t;x,y)| \le C_2 t^{-\frac{1}{\alpha}}\sum_{j\in\Z^d}\left( \frac{t}{|x-y+j|^{d+\alpha}}\wedge t^{-\frac{d}{\alpha}} \right) },
  \end{gather*}
  for all $x,y\in\R^d$ and $t\in(0,T]$, where $C_1>1, C_2>0$ are two constants depending on $d,\alpha, \|b\|_0, h_0,h_1$. The constants $h_0,h_1$ are related to the function $h$ as in Lemma \ref{regularity_h}.
\end{lemma}

\begin{remark}\label{rep-Feller}
  (1). Combining Lemma \ref{Feller} and Lemma \ref{martingale-prob}, we see that the Feller semigroup $\{P_t\}_{t\ge0}$ generated by the closure of $\L^\alpha$ has the representation
  \begin{equation*}
    P_t f(x)=\int_{\T^d} f(y) p(t;x,y) dy, \quad f\in \C(\T^d),
  \end{equation*}
  and the following gradient estimate holds
  \begin{equation}\label{gradient-estimate}
    \begin{split}
       |\nabla P_tf(x)| \le &\ C_2 \|f\|_0 t^{-\frac{1}{\alpha}} \textstyle{ \int_{\T^d} \sum_{j\in\Z^d}\left( \frac{t}{|y+j|^{d+\alpha}}\wedge t^{-\frac{d}{\alpha}} \right) dy } \\
         = &\ C_2 \|f\|_0 t^{-\frac{1}{\alpha}} \textstyle{ \int_{\R^d} \left( \frac{t}{|y|^{d+\alpha}}\wedge t^{-\frac{d}{\alpha}} \right) dy } \le \ C_2 \textstyle{ \left(1+\frac{1}{\alpha}\right)} \|f\|_0 t^{-\frac{1}{\alpha}}.
    \end{split}
  \end{equation}

  (2). Denote the formal generator of $\tilde X^{x,\e}$ by $\tilde\L^\alpha_\e$, i.e.,
  \begin{equation}\label{L-e}
    \begin{split}
       \tilde\L_\e^\alpha f(x) :=&\ \int_{\Ro d} \textstyle{ \left[ f\left( x+\DivEps 1\sigma\left(x,\e y\right)\right)-f(x)- \DivEps 1 \sigma^i\left(x,\e y\right) \partial_i f(x) \ind_{B}(y) \right] } \nu^\alpha(dy) \\
       & + \left[ b^i\left(x\right) + \e^{\alpha-1} c^i\left(x\right) \right] \partial_i f(x),\quad x\in\R^d.
    \end{split}
  \end{equation}
  Then Lemma \ref{Feller} and \ref{martingale-prob} still hold true with $\tilde\L^\alpha_\e$ in place of $\L^\alpha$.
\end{remark}

As a corollary, the solution of equation \eqref{resovent} can be represented in terms of a semigroup, and satisfies a finer estimate.

\begin{corollary}\label{precise-estimates}
  For any $\kappa>0$ and $f\in\C^\beta(\T^d)$, the unique solution $u_\kappa$ of equation \eqref{resovent} admits the representation
  \begin{equation}\label{rep_u_lambda}
    u_\kappa(x)=\int_0^\infty e^{-\kappa t} P_t f(x) dt,
  \end{equation}
  where $\{P_t\}_{t\ge0}$ is the Feller semigroup generated by the closure of $\L^\alpha$, and the integral on the right hand side converges. Moreover, there exists a constant $C>0$ independent of $u,f,b,\kappa$ such that
  \begin{equation}\label{energy3}
    \kappa\|u_\kappa\|_0+\kappa^{\frac{\alpha+\beta-1}{\alpha}} \|\nabla u_\kappa\|_0 + [\nabla u_\kappa]_{\alpha+\beta-1}\le C\|f\|_\beta.
  \end{equation}
\end{corollary}

\begin{proof}
  Proposition \ref{res_prob} tells that the interval $(0,+\infty)$ is contained in the resolvent set of $\L^\alpha$. Then by the integral representation of the resolvent (see \cite[Theorem II.1.10.(ii)]{EN00}), we arrive at
  \begin{equation*}
    u_\kappa=(\kappa-\L^\alpha)^{-1}f = \lim_{t\to\infty} \int_0^t e^{-\kappa s}P_s fds,
  \end{equation*}
  where the limit is taken in $(\C(\T^d),\|\cdot\|_0)$. The representation \eqref{rep_u_lambda} then follows. Now thanks to the gradient estimate \eqref{gradient-estimate} and representation \eqref{rep_u_lambda}, the estimate \eqref{energy3} is then obtained by the same argument as the proof of \cite[Theorem 3.3, Part I]{Pri12}.
\end{proof}

In the next section, we will remove the large jumps from the SDEs and study their well-posedness by Zvonkin's transform. Thus we consider the following operator, which is a ``flat'' version of $\L^{\alpha}$:
\begin{equation}\label{L-flat}
  \L^{\alpha,\flat}f(x) = \int_{\Bo} \big[ f( x+\bar\sigma(x,y))-f(x) - \bar\sigma^i(x,y) \partial_i f(x) \big] \nu^\alpha(dy) + b^i(x)\partial_i f(x).
\end{equation}
We have the following regularity result for $\L^{\alpha,\flat}$.

\begin{corollary}\label{no-large-jump}
  There exists a constant $\kappa_*>0$ such that for any $\kappa>\kappa_*$ and $f\in\C^\beta(\T^d)$, there exists a unique solution $u=u^\flat_\kappa\in\C^{\alpha+\beta}(\T^d)$ to the equation
  \begin{equation}\label{small-jump}
    \kappa u-\L^{\alpha,\flat} u = f.
  \end{equation}
  In addition, there exists a constant $C>0$ independent of $u,f,b,\kappa$, such that for any $\kappa>\kappa_*$,
  \begin{equation}\label{energy4}
    (\kappa-\kappa_*)\|u^\flat_\kappa\|_0+(\kappa-\kappa_*)^{\frac{\alpha+\beta-1}{\alpha}} \|\nabla u^\flat_\kappa\|_0 + [\nabla u^\flat_\kappa]_{\alpha+\beta-1}\le C\|f\|_\beta.
  \end{equation}
\end{corollary}

\begin{proof}
  To obtain the a priori estimate \eqref{energy4}, we rewrite the equation \eqref{small-jump} in the form
  \begin{equation*}
    \kappa u-\L^{\alpha} u = f - \int_{B^c}[u(x+\bar\sigma(x,y))-u(x)]\nu^\alpha(dy).
  \end{equation*}
  The estimate \eqref{energy3} implies that
  \begin{equation*}
    \kappa\|u\|_0+\kappa^{\frac{\alpha+\beta-1}{\alpha}} \|\nabla u\|_0 + [\nabla u]_{\alpha+\beta-1}\le C(\|f\|_\beta+2\nu^\alpha(B^c)\|u\|_\beta).
  \end{equation*}
  It is easy to see that there exists $\delta>0$ such that
  \begin{equation*}
    \sup_{|x-y|<\delta} \frac{|u(x)-u(y)|}{|x-y|} \le 2\|\nabla u\|_0,
  \end{equation*}
  and then
  \begin{equation*}
    \|u\|_\beta \le \sup_{|x-y|<\delta} \frac{|u(x)-u(y)|}{|x-y|}|x-y|^{1-\beta} + \sup_{|x-y|\ge\delta} \frac{|u(x)-u(y)|}{|x-y|^\beta} \le 2\delta^{1-\beta}\|\nabla u\|_0 + 2\delta^{-\beta}\|u\|_0.
  \end{equation*}
  Combining these together, we get
  \begin{equation*}
    \big(\kappa-4C\delta^{-\beta}\nu^\alpha(B^c)\big)\|u\|_0 +\big(\kappa^{\frac{\alpha+\beta-1}{\alpha}} - 4C\delta^{1-\beta}\nu^\alpha(B^c) \big)\|\nabla u\|_0 + [\nabla u]_{\alpha+\beta-1}\le C\|f\|_\beta.
  \end{equation*}
  Then \eqref{energy4} follows by choosing $\kappa_*=4C\delta^{-\beta}\nu^\alpha(B^c) \vee \big( 4C\delta^{1-\beta}\nu^\alpha(B^c) \big)^\frac{\alpha}{\alpha+\beta-1}$.

  Now define a family of operators by
  \begin{equation*}
    \L_\theta^\flat=\L^{\alpha,\flat}+\theta \int_{B^c}[u(x+\bar\sigma(x,y))-u(x)]\nu^\alpha(dy).
  \end{equation*}
  Then $\L_1^\flat=\L^\alpha, \L_0^\flat=\L^{\alpha,\flat}$. The well-posedness of equation \eqref{small-jump} follows from the method of continuity and the a priori estimate \eqref{energy4}, just as in the proof of Proposition \ref{res_prob}.
\end{proof}

\section{SDEs with multiplicative stable L\'evy noise}\label{SDEs}

The goal of this section is to study the strong well-posedness of SDEs \eqref{Xe_tilde} and \eqref{X_tilde}, as well as the invariance and ergodicity of the solution processes $\tilde X^{x,\epsilon}$ for each $\e \ge0$. As corollaries, we also obtain the Feynman-Kac formula and the well-posedness of nonlocal Poisson equation without zeroth-order term, which will be used to study homogenization in the next two sections.

\subsection{Strong well-posedness of SDEs}\label{subsec:4-1}

We only consider the strong well-posedness for SDE \eqref{X_tilde} since \eqref{Xe_tilde} has the same form. As we have seen in Lemma \ref{martingale-prob}, the existence and uniqueness hold for the martingale problem for $(\L^\alpha,\delta_x)$. Meanwhile, it is known that the martingale solution for $(\L^\alpha,\delta_x)$ is
equivalent to the weak solution of SDE \eqref{X_tilde}, see \cite[Theorem 2.3, Corollary 2.5]{Kur11}. Thus, the existence and uniqueness of weak solution hold for SDE \eqref{X_tilde}.

Moreover, utilizing the fact shown in \cite[Theorem 1.2]{BLP15} that the weak existence and pathwise uniqueness for SDE \eqref{X_tilde} imply strong existence, we only need to prove the pathwise uniqueness. The key is to reduce the SDE \eqref{X_tilde}, whose coefficients have low regularity, to an SDE with Lipschitz coefficients by using Zvonkin's transform.

For $\kappa>\kappa_*$, let $\hat b_\kappa\in\C^{\alpha+\beta}(\T^d)$ be the solution of
\begin{equation*}
  \kappa \hat b_\kappa-\L^{\alpha,\flat} \hat b_\kappa = b,
\end{equation*}
where $\L^{\alpha,\flat}$ is the operator in \eqref{L-flat}. The existence and uniqueness of solution $\hat b_\kappa$ is ensured by Corollary \ref{no-large-jump}. Define a map $\Phi_\kappa:\R^d\to\R^d$ by
\begin{equation*}
  \Phi_\kappa(x)=x+\hat b_\kappa(x).
\end{equation*}
Then $\Phi_\kappa$ is of class $\C^{\alpha+\beta}$. Moreover, we have the following lemma, whose proof is standard and left into Appendix \ref{app}.

\begin{lemma}\label{Phi}
  For $\kappa>0$ large enough, the map $\Phi_\kappa:\R^d\to\R^d$ is a $C^1$-diffeomorphism and its inverse $\Phi_\kappa^{-1}$ is also of class $\C^{\alpha+\beta}$.
\end{lemma}

To solve SDE \eqref{X_tilde}, by a standard interlacing technique (cf. \cite[Section 6.5]{App09} or \cite[Theorem IV. 9.1]{IW14}), it suffices to solve the following SDE with no jumps greater than one:
\begin{equation*}
  \tilde X^{x,\flat}_t = x+\int_0^t b(\tilde X^{x,\flat}_{s})ds + \int_0^t\int_B \bar\sigma(\tilde X^{x,\flat}_{s-},y) \tilde N^\alpha(dy,ds).
\end{equation*}

Now fix $\kappa>0$ large enough such that the conclusions in Lemma \ref{Phi} hold. We introduce Zvonkin's transform
\begin{equation*}
  \tilde X^{*}_t=\Phi_\kappa(\tilde X^{x,\flat}_t).
\end{equation*}
Then by applying It\^o's formula, we have
\begin{equation}\label{Zvonkin}
  \tilde X^{*}_t = \Phi_\kappa(x) + \int_0^t b^{*}(\tilde X^{*}_{s}) ds + \int_0^t\int_B \bar\sigma^{*}(\tilde X^{*}_{s-},y) \tilde N^\alpha(dy,ds),
\end{equation}
where
\begin{equation*}
  b^{*}(x)=\kappa\hat b_\kappa(\Phi_\kappa^{-1}(x)),
\end{equation*}
\begin{equation*}
  \bar\sigma^{*}(x,y)=\hat b_\kappa(\Phi_\kappa^{-1}(x)+ \bar\sigma(\Phi_\kappa^{-1}(x),y)) - \hat b_\kappa(\Phi_\kappa^{-1}(x)) + \bar\sigma(\Phi_\kappa^{-1}(x),y).
\end{equation*}

\begin{proposition}\label{strong-sol}
  For each $x\in\R^d$, there is a unique strong solution $\tilde X^x=\{\tilde X_t^x\}_{t\ge0}$ to SDE \eqref{X_tilde}.
\end{proposition}

\begin{proof}
  By the above argument, we only need to prove the pathwise uniqueness for SDE \eqref{Zvonkin}. First of all, we have, for any $x,x_1,x_2\in\R^d$,
  \begin{equation}\label{b*}
    \big|b^{*}(x_1)-b^{*}(x_2)\big| \le C(\|\hat b_\kappa\|_1,\|\Phi_\kappa^{-1}\|_1)|x_1-x_2|,
  \end{equation}
  Note that for $\gamma\in(0,1)$, $f\in\C_b^{1+\gamma}(\R^d)$, $x,u,v\in\R^d$, there exists a constant $C>0$ such that
  \begin{equation*}
    |f(u+x)-f(u)-f(v+x)-f(v)| \le C \|f\|_{1+\gamma} |u-v||x|^\gamma,
  \end{equation*}
  the proof can be found in \cite[Theorem 5.1.(c)]{Bas09}. Then for any $x_1,x_2$,
  \begin{equation*}
    \begin{split}
       &\ \left|\bar\sigma^{*}(x_1,y)- \bar\sigma^{*}(x_2,y)\right| \\
       \le&\ \big|\hat b_\kappa(\Phi_\kappa^{-1}(x_1)+ \bar\sigma(\Phi_\kappa^{-1}(x_1),y)) - b_\kappa(\Phi_\kappa^{-1}(x_1)) \\
         &\ - \hat b_\kappa(\Phi_\kappa^{-1}(x_2)+ \bar\sigma(\Phi_\kappa^{-1}(x_1),y)) + b_\kappa(\Phi_\kappa^{-1}(x_2))\big| \\
         &\ + \big|\hat b_\kappa(\Phi_\kappa^{-1}(x_2)+ \bar\sigma(\Phi_\kappa^{-1}(x_1),y)) - \hat b_\kappa(\Phi_\kappa^{-1}(x_2)+ \bar\sigma(\Phi_\kappa^{-1}(x_2),y))\big| \\
         &\ + \big|\bar\sigma(\Phi_\kappa^{-1}(x_1),y)- \bar\sigma(\Phi_\kappa^{-1}(x_2),y)\big| \\
         \le&\ C \|\hat b_\kappa\|_{\alpha+\beta} \left| \Phi_\kappa^{-1}(x_1)-\Phi_\kappa^{-1}(x_2) \right| \left| \bar\sigma(\Phi_\kappa^{-1}(x_1),y) \right|^{\alpha+\beta-1} \\
         &\ +(\|\nabla \hat b_\kappa\|_0+1)\left| \bar\sigma(\Phi_\kappa^{-1}(x_1),y)- \bar\sigma(\Phi_\kappa^{-1}(x_2),y) \right| \\
         \le&\ C \left(\|\hat b_\kappa\|_{\alpha+\beta}, \|\Phi_\kappa^{-1}\|_1,  \|\phi\|_{L^\infty} \right) |x_1-x_2|(|y|^{\alpha+\beta-1}+|y|).
    \end{split}
  \end{equation*}
  where we have used the regularity condition for $\bar\sigma$ in Assumption \ref{sigma}, and $\phi$ is the positive bounded function in the growth condition in that assumption. Noting that $2(\alpha+\beta-1)>\alpha$ by Assumption \ref{regular}, we arrive at
  \begin{equation}\label{sigma*}
    \int_B \left|\bar\sigma^{*}(x_1,y)- \bar\sigma^{*}(x_2,y)\right|^2 \nu_\alpha(dy) \le C\left(\|\hat b_\kappa\|_{\alpha+\beta}, \|\Phi_\kappa^{-1}\|_1,  \|\phi\|_{L^\infty} \right)|x_1-x_2|^2.
  \end{equation}
  The pathwise uniqueness of SDE \eqref{Zvonkin} follows from \eqref{b*}, \eqref{sigma*} and the classical result \cite[Theorem 4.9.1]{IW14}. The proof is complete.
\end{proof}

The following corollary is just a summarizing of above results for the process $\tilde X^x$, the proof is left into Appendix \ref{app}.
\begin{corollary}\label{strong-Markov}
  The solution process $\tilde X^x$
  is a Feller process with generator the closure of $(\L^\alpha,\C^{\alpha+\beta}(\T^d))$.
  In particular, $\tilde X^x$ is a strong Markov process.
\end{corollary}

\begin{remark}
  The Feller semigroup $\{P_t\}_{t\ge0}$ in Lemma \ref{Feller} is the semigroup associated with the solution process $\tilde X^x$, that is,
  \begin{equation*}
    P_tf(x)=\E(f(\tilde X^x_t)), \quad f\in \C(\T^d).
  \end{equation*}
\end{remark}

As a consequence of the Feller property, we can obtain the well-posedness of the parabolic nonlocal PDE and the corresponding Feynman-Kac representation. See \cite{PR14} for the classical version for second order PDE.
\begin{proposition}\label{wellposed-parabolic}
  The parabolic nonlocal PDE
  \begin{equation*}
    \begin{cases}
      \frac{\partial u}{\partial t}(t,x) = \L^\alpha u(t,x) +g(x) u(t,x), & t>0,x\in\R^d, \\
      u(0,x) = u_0(x), & x\in\R^d,
    \end{cases}
  \end{equation*}
  admits a unique mild solution in the sense that $\int_0^t u(s) ds\in D(\L^\alpha)$ for all $t\ge0$ and
  \begin{equation*}
    u(t) = u_0 +\L^\alpha \int_0^t u(s) ds + g\int_0^t u(s) ds.
  \end{equation*}
  Moreover, the unique solution has the following Feynman-Kac representation
  \begin{equation*}
    u(t,x)=\E\left[u_0(\tilde X^x_t)\exp\left(\int_0^t g(\tilde X^x_s) ds\right)\right].
  \end{equation*}
\end{proposition}

\begin{proof}
  Choose $G>0$ large enough such that $\|g\|_0<G$. Define
  \begin{equation*}
    P_t^g f(x)=\E\left[f(\tilde X^x_t)\exp\left(\int_0^t g(\tilde X^x_s) ds- Gt\right)\right], \quad f\in \C(\T^d).
  \end{equation*}
  Then by an argument similar to that used in \cite[Section 6.7.2]{App09}, one can show that $\{P_t^g\}_{t\ge0}$ is a Feller semigroup with generator the closure of $(\L^\alpha+g-G,\C^{\alpha+\beta}(\T^d))$. This yields that $\{e^{Gt}P_t^g\}_{t\ge0}$ is a $C_0$-semigroup on $\C(\T^d)$ with generator the closure of $(\L^\alpha+g,\C^{\alpha+\beta}(\T^d))$. Note that $u_0\in\C(\T^d)$. Now applying the classic result \cite[Proposition II.6.4]{EN00} in the theory of $C_0$-semigroups, we conclude that the parabolic nonlocal PDE admits a unique mild solution, which can be given by the orbit map $u(t)=e^{Gt}P_t^g u_0$. The desired conclusions follow immediately.
\end{proof}

\subsection{Invariance and ergodicity}

Now we deal with the invariant measures and ergodicity of SDEs. From now on, Assumptions \ref{coef1}-\ref{scaling} are set in force.

By the discussion in previous subsection, SDE \eqref{Xe} and SDE \eqref{Xe_tilde} for each $\e\ge0$, admit the unique strong solutions $\tilde X^{x}$ and $\tilde X^{x,\e}$, which are all $\T^d$-valued Feller processes. Denote by $p^\epsilon(t;x,y)$ the transition probability density of $\tilde X^{x,\epsilon}$, by $\{P^\epsilon_t\}_{t\ge 0}$ the associated Feller semigroup. Recall that we denote $\tilde X^{x,0}:=\tilde X^x$ for notational simplicity.

\begin{lemma}\label{inv}
  For each $0\le\epsilon\le1$, the process $\tilde X^{x,\epsilon}$ possesses a unique invariant distribution $\mu_\epsilon$ on $\T^d$. Moreover, there exist positive constants $C$ and $\rho$, depending only on $d,\alpha, \|b\|_0,\|c\|_0, h_0,h_1$, such that for any periodic bounded Borel function $f$ on $\R^d$ (i.e., $f$ is Borel bounded on $\T^d$),
  \begin{equation*}
    \sup_{x\in\T^d} \left| P^\e_t f(x)-\int_{\T^d}f(y)\mu_\epsilon(dy) \right| \le C\|f\|_0e^{-\rho t}
  \end{equation*}
  for every $t\ge0$.
\end{lemma}

\begin{proof}
  One can find a version of Doeblin-type result in \cite[Theorem 3.3.1, 3.3.2]{BLP78}, which states that for a Markov process with transition probability densities bounded from below by a positive constant, it has a unique invariant probability measure and the associated semigroup converges exponentially fast. Therefore, it is enough to ensure that the transition probability density $p^\e(1;x,y)$ is bounded from below by a positive constant, which follows immediately from the density estimates in Lemma \ref{martingale-prob}.
  Moreover, one can also realize from \cite[Theorem 3.3.1, 3.3.2]{BLP78} that the two constants $C$ and $\rho$ are only related to the lower bound of $p^\e(1;x,y)$. The only thing left is to verify the dependence of the constants $C$ and $\rho$. In particular, they can be chosen uniformly in $\e\in[0,1]$.

  In the case $\e=0$, the lower bound of $p(1;x,y)$ depends on $d,\alpha, \|b\|_0, \bar h_0,\bar h_1$ by Lemma \ref{martingale-prob}, where $\bar h_0$ and $\bar h_1$ are the constants associated to $\bar\sigma$ as indicated in Lemma \ref{regularity_h} and Remark \ref{sigma-conv}. We turn to the case that $\e\in(0,1]$. Since the generator of each semigroup $\{P^\e_t\}$ is $\tilde\L^\alpha_\e=\A^{\frac{1}{\e}\sigma_\e,\nu^\alpha}+ (b+\e^{\alpha-1}c)\cdot\nabla$ as in \eqref{L-e}, where $\sigma_\epsilon(x,y):=\sigma\left(\frac{x}{\e},y\right)$, it is easy to derive that the function $h^\e$ associated with $\frac{1}{\e}\sigma_\epsilon$ as in \eqref{h} is given by
  \begin{equation*}
    h_\e(x,z)=|\det \nabla_z\tau(x,z)| \frac{|\e z|^{d+\alpha}}{|\tau(x,\e z)|^{d+\alpha}},
  \end{equation*}
  Hence, the constants $h^\e_0$ and $h^\e_1$ associated to $h_\e$ as in Lemma \ref{regularity_h} can be chosen to equal to $h_0$ and $h_1$, respectively. We conclude that when $\e\in(0,1]$, the constant $C_1$ associated to $p^\e(t;x,y)$ in Lemma \ref{martingale-prob} are related to $d,\alpha, \|b+\e^{\alpha-1}c\|_0, h_0,h_1$. Hence, for $\e\in[0,1]$, constants $C$ and $\rho$ can be chosen to depend only on $d,\alpha, \|b\|_0,\|c\|_0, h_0,h_1,\bar h_0,\bar h_1$. In particular, they are independent of $\e\in[0,1]$.
\end{proof}

%


Denote by $\mu=\mu_0$ the unique invariant probability measure for the limit process $\tilde X^x_t$ in \eqref{X_tilde}. Then we can prove the following lemma.

\begin{lemma}\label{conv_mu}
  As $\epsilon\to0$, we have $\mu_\epsilon\to\mu$ weakly.
\end{lemma}

\begin{proof}
  Using the same argument as the proof of \cite[Lemma 2.4]{HP08}, and noting that the tightness of the family $\{\mu_\e\}_{\e>0}$ is automatic due to the compactness of $\T^d$, it suffices to prove that $P_t^\e f\to P_tf$ in $\C(\T^d)$ as $\e\to 0$ for any $f\in\C(\T^d)$ and $t\ge0$. By Lemma \ref{martingale-prob} and Remark \ref{rep-Feller}.(2), we know that $\C^{\alpha+\beta}(\T^d)$ is a core for $\L^\alpha$ and each $\tilde\L^\alpha_\e$, $\e>0$. Fix an arbitrary $f\in\C^{\alpha+\beta}(\T^d)$,
  \begin{equation*}
    \begin{split}
       |\tilde\L^\alpha_\e &f(x)-\L^\alpha f(x)| \le \e^{\alpha-1} \left| c(x)\cdot\nabla f(x) \right| + \int_{B^c} \left| \textstyle{ f\left( x+\DivEps 1\sigma(x,\e y)\right) } - f(x+\bar\sigma(x,y)) \right|\nu^\alpha(dy) \\
         & + \int_\Bo \big| \textstyle{ f\left( x+\DivEps 1\sigma(x,\e y)\right) } - f(x+\bar\sigma(x,y)) - \textstyle{ \left( \DivEps 1\sigma(x,\e y) - \bar\sigma(x,y) \right) } \cdot\nabla f(x) \big|\nu^\alpha(dy) \\
         \le&\ \e^{\alpha-1}\|c\|_0 \|f\|_1 + \|\nabla f\|_0 \int_{B^c} \textstyle{ \left| \DivEps 1\sigma(x,\e y) - \bar\sigma(x,y) \right| } \nu^\alpha(dy) \\
         & + {\textstyle{ \frac{1}{\alpha+\beta} }} \|f\|_{\alpha+\beta} \int_{\Bo} \left| \textstyle{ \DivEps 1\sigma(x,\e y) } - \bar\sigma(x,y) \right|^{\alpha+\beta} \nu^\alpha(dy).
    \end{split}
  \end{equation*}
  By the growth condition for $\sigma$ and $\bar\sigma$ as indicated in Assumption \ref{sigma} and Remark \ref{sigma-conv},
  \begin{equation*}
    \int_{B^c} \sup_{x\in\T^d} {\textstyle{ \left| \DivEps 1\sigma(x,\e y) - \bar\sigma(x,y) \right|}}\nu^\alpha(dy) \le 2\|\phi\|_{L^\infty} \int_{B^c} |y| \nu^\alpha(dy) < \infty,
  \end{equation*}
  since $\alpha>1$. Similarly,
  \begin{equation*}
    \int_\Bo \sup_{x\in\T^d} {\textstyle{ \left| \DivEps 1\sigma(x,\e y) - \bar\sigma(x,y) \right|}}^{\alpha+\beta} \nu^\alpha(dy) \le (2\|\phi\|_{L^\infty})^{\alpha+\beta} \int_\Bo |y|^{\alpha+\beta} \nu^\alpha(dy) < \infty.
  \end{equation*}
  Hence, the dominated convergence and Assumption \ref{scaling} implies that $|\tilde\L^\alpha_\e f(x)-\L^\alpha f(x)|$ converges to zero uniformly in $x$, as $\e\to0$. That is,
  \begin{equation*}
    \|\tilde\L^\alpha_\e f-\L^\alpha f\|_0 \to 0, \quad\text{as } \e\to0.
  \end{equation*}
  Now using the Trotter-Kato approximation theorem (see \cite[Theorem III.4.8]{EN00}), $P_t^\e f\to P_tf$ in $\C(\T^d)$ as $\e\to 0$ for all $f\in\C(\T^d)$, uniformly for $t$ in compact intervals.
\end{proof}

Now we combine Lemma \ref{inv} and Lemma \ref{conv_mu} to get the following ergodic theorem.

\begin{proposition}\label{ergodic_thm}
  Let $f$ be a bounded Borel function on $\T^d$. Then for any $t>0$,
  \begin{equation*}
    \int_0^t \left|f {\textstyle{ \left( \frac{X_s^{x,\epsilon}}{\epsilon} \right) }} - \int_{\T^d}f(x)\mu(dx) \right| ds \to 0
  \end{equation*}
  in probability, as $\epsilon\to 0$.
\end{proposition}

\begin{proof}
  We follow the lines of \cite[Proposition 2.4]{Par99}. For $\e>0$, $0\le s<t$, let $\bar f$  be a bounded measurable function on $\T^d$ satisfying $\int_{\T^d}\bar f(x)\mu_\e(dx)=0$. By Lemma \ref{conv_mu}, it suffices to prove that $\int_0^t |\bar f(X^\e_s/\e)| ds \to 0$ in $L^2(\Omega,\P)$. Using Lemma \ref{inv}, we have
  \begin{equation*}
    \E \left[|\bar f(\tilde X^{x,\e}_t)|\Big| \tilde X^{x,\e}_s\right] = \int_{\T^d} |\bar f(y)|\left[p^\e(t-s,\tilde X_s^{x,\e},y)dy-\mu_\e(dy)\right] \le C\|\bar f\|_0 e^{-\rho(t-s)}.
  \end{equation*}
  By the Markov property,
  \begin{equation*}
    \E|\bar f(\tilde X^{x,\e}_s) \bar f(\tilde X^{x,\e}_t)| = \E\left[|\bar f(\tilde X^{x,\e}_s)|\E \left(|\bar f(\tilde X^{x,\e}_t)|\Big| \tilde X^{x,\e}_s\right)\right] \le C\|\bar f\|_0^2 e^{-\rho(t-s)}.
  \end{equation*}
  Hence,
  \begin{equation*}
    \begin{split}
       \E\left[\left( \int_0^t |\bar f(X^{x,\e}_s/\e)| ds \right)^2\right] &= \e^{2\alpha}\int_0^{\e^{-\alpha}t} \int_0^{r} \E|\bar f(\tilde X^{x,\e}_s) \bar f(\tilde X^{x,\e}_r)| ds dr \\
         &\le 2C\e^{2\alpha} \|f\|_0^2 \int_0^{\e^{-\alpha}t} \int_0^{r} e^{-\rho(r-s)} dsdr \\
         &= 2C\e^{2\alpha} \|f\|_0^2 \rho^{-2} (-1+\rho\e^{-\alpha}t + e^{-\rho\e^{-\alpha}t}) \\
         &\to 0,
    \end{split}
  \end{equation*}
  as $\e\to0$. The results follow.
\end{proof}

For every $\gamma>0$, denote by $\C_\mu^\gamma(\T^d)$ the class of all $f\in\C^\gamma(\T^d)$ which are \emph{centered} with respect to the invariant measure $\mu$ in the sense that $\int_{\T^d}f(x)\mu(dx)=0$. It is easy to check that $\C_\mu^\gamma(\T^d)$ is closed, and hence a sub-Banach space of $\C^\gamma(\T^d)$ under the norm $\|\cdot\|_\gamma$.

Thanks to Proposition \ref{res_prob} and Lemma \ref{inv}, we can use the Fredholm alternative to obtain the solvability of the following Poisson equation without zeroth-order term in the smaller space $\C_\mu^{\alpha+\beta}(\T^d)$,
\begin{equation}\label{Poisson}
  \L^\alpha u + f = 0,
\end{equation}
for $f\in\C_\mu^\beta(\T^d)$. Before that, we need some lemmas that are straightforward consequences of the classical theory of $C_0$-semigroups. We provide their proofs in Appendix \ref{app}.

\begin{lemma}\label{restriction}
  The restrictions $\{P_t^\mu := P_t|_{\C_\mu(\T^d)}\}_{t\ge0}$ form a $C_0$-semigroup on the Banach space $(\C_\mu(\T^d),\|\cdot\|_0)$, with generator given by $\L^\alpha_\mu f:= \L^\alpha f$, $D(\L^\alpha_\mu) := \C_\mu^{\alpha+\beta}(\T^d)$.
\end{lemma}

\begin{lemma}\label{res_prob-center}
  If $f\in\C_\mu^\beta(\T^d)$, then the unique solution $u_\kappa$ of \eqref{resovent} is of class $\C_\mu^{\alpha+\beta}(\T^d)$, for any $\kappa>0$.
\end{lemma}

The following theorem will solve the well-posedness of equation \eqref{Poisson}, which is more general than the results in \cite[Proposition 3]{Fra07}. We formulate it as follows, referring to \cite[Theorem 1]{PV01} for the classical version for second order partial differential operators.

\begin{proposition}\label{wellposed-Poisson}
  For any $f\in\C_\mu^\beta(\T^d)$, there exists a unique solution in $\C_\mu^{\alpha+\beta}(\T^d)$ to the equation \eqref{Poisson}, which satisfies the estimate
  \begin{equation}\label{energy2}
    \|u\|_{\alpha+\beta} \le C (\|u\|_0+\|f\|_\beta),
  \end{equation}
  where $C=C(\|b\|_\beta)$ is a positive constant. Moreover, the unique solution admits the representation
  \begin{equation}\label{rep_u}
    u(x)=\int_0^\infty P_t f(x) dt.
  \end{equation}
\end{proposition}

\begin{proof}
  The a priori estimate \eqref{energy2} is also from \citep[Theorem 7.1]{Bas09}.

  First, we show that if the equation has a solution $u\in\C_\mu(\T^d)$ for $f\in\C_\mu^\beta(\T^d)$, then $u$ must have the representation \eqref{rep_u}, this also implies the uniqueness. By the exponential ergodicity result in Lemma \ref{inv}, we have $\|P_t^\mu f\|_0\le C\|f\|_0 e^{-\rho t}$ for any $f\in\C_\mu(\T^d)$ and $t\ge0$. This yields that, using \cite[Theorem II.1.10.(ii)]{EN00} as in the proof of Corollary \ref{precise-estimates}, the set $\{z\in\Cp|\text{Re}z>-\rho\}$ is contained in the resolvent set of $\L^\alpha_\mu$. Noting that $u=(0-\L^\alpha_\mu)^{-1}f$, the representation and uniqueness follow.

  Now we prove the existence. Let $\kappa_0$ be a fixed positive constant. Thanks to Lemma \ref{res_prob-center}, the linear map $\kappa_0-\L^\alpha:\C_\mu^{\alpha+\beta}(\T^d)\to \C_\mu^{\beta}(\T^d)$ is invertible. Furthermore, by virtue of Lemma \ref{MP} and the energy estimate \eqref{energy1}, together with the compact embedding $\C_\mu^{\alpha+\beta}(\T^d)\subset\C_\mu^{\beta}(\T^d)$ (see, for instance, \cite[Lemma 6.36]{GT01}), the resolvent $\mathcal R_{\kappa_0}:=(\kappa_0-\L^\alpha)^{-1}$ is compact from $\C_\mu^{\beta}(\T^d)$ to $\C_\mu^{\beta}(\T^d)$. Consider then the equation
  \begin{equation}\label{Fred}
    u-\kappa_0\mathcal R_{\kappa_0}u=\mathcal R_{\kappa_0}f, \quad f\in\C_\mu^{\beta}(\T^d),
  \end{equation}
  Then the Fredholm alternative (see \cite[Section 5.3]{GT01}) implies that the equation \eqref{Fred} always has a unique solution $u\in\C_\mu^{\beta}(\T^d)$ provided the homogeneous equation $u-\kappa_0\mathcal R_{\kappa_0}u=0$ has only the trivial solution $u=0$.

  To rephrase these statements in terms of the Poisson equation \eqref{Poisson}, we observe first that since $\mathcal R_{\kappa_0}$ maps $\C_\mu^{\beta}(\T^d)$ onto $\C_\mu^{\alpha+\beta}(\T^d)$, any solution $u\in\C_\mu^{\beta}(\T^d)$ of \eqref{Fred} must also belong to $\C_\mu^{\alpha+\beta}(\T^d)$. Hence, operating on \eqref{Fred} with $\kappa_0-\L^\alpha$ we obtain
  \begin{equation*}
    -\L^\alpha u= (\kappa_0-\L^\alpha)(u-\kappa_0\mathcal R_{\kappa_0}u)=f.
  \end{equation*}
  Thus, the solutions of \eqref{Fred} are in one-to-one correspondence with the solutions of the Poisson equation \eqref{Poisson}. Consequently, \eqref{Fred} has a unique solution in $\C_\mu^{\alpha+\beta}(\T^d)$ if we can show that the homogeneous equation $\L^\alpha u=0$ has only the zero solution, while the latter follows from the representation \eqref{rep_u}.
\end{proof}

\begin{remark}
  The assumption that $f$ is centered with respect to $\mu$ in Proposition \ref{wellposed-Poisson} is necessary. To see this informally, let's recall the Riesz-Schauder theory for compact operators (cf. \cite[Theorem X.5.3]{Yos95}). The equation \eqref{Fred} admits a solution $u\in\C(\T^d)$ if and only if $\mathcal R_\kappa f\in\text{Ker}(I^*-\kappa\mathcal R_\kappa^*)^\bot$, where the superscript $*$ denotes the adjoint of operators. This is equivalent to say that the equation \eqref{Poisson} admits a solution $u\in\C(\T^d)$ if and only if $f\in\text{Ker}(L^{\alpha,*})^\bot$. On the other hand, we have $\mu\in\text{Ker}(L^{\alpha,*})$ since $\mu$ is the invariant measure with respect to $\{P_t\}_{t\ge0}$. Thus a necessary condition for the existence of \eqref{Poisson} is $\langle\mu,f\rangle=0$, regarding $\mu$ as an element in the dual space of $\C(\T^d)$.
\end{remark}

\section{Homogenization results}\label{Homogenization}

\subsection{Homogenization of SDEs}

The aim of this subsection is to show the homogenization result of the solutions $X^{x,\e}$ of SDEs \eqref{Xe}. We will let Assumptions \ref{coef1}-\ref{center} all hold true in this and next subsection.

It is quite natural to get rid of the drift term involving $\frac{1}{\e^{\alpha-1}}$ in \eqref{Xe}. For this purpose, we again use Zvonkin's transform,
\begin{equation*}
  \hat X_t^{x,\epsilon}:=X_t^{x,\epsilon}+\epsilon \textstyle{ \left(\hat b\left(\frac{X_t^{x,\epsilon}}{\epsilon}\right)-\hat b\left(\frac{x}{\epsilon}\right)\right) } ,
\end{equation*}
where $\hat b$ is the solution of the Poisson equation
\begin{equation}\label{b_hat}
  \L^\alpha \hat b + b = 0,
\end{equation}
with the linear operator $\L^\alpha$ given by \eqref{L}. Note that the transform here is slightly different from that used in Section \ref{SDEs}. Due to Proposition \ref{wellposed-Poisson}, $\hat b\in\C_\mu^{\alpha+\beta}(\T^d)$ is uniquely determined under Assumptions \ref{center}.

We also need an elementary lemma. The proof is elementary 
and shall be omitted.

\begin{lemma}\label{difference-estimate-2}
  Let $0<\gamma\le1$ and $f\in\C_b^{1+\gamma}(\R^d)$. For any $x,u,v\in\R^d$, it holds that
  \begin{equation*}
    |f(x+u)-f(x+v)-(u-v)\cdot\nabla f(x)|\le \textstyle{ \frac{1}{1+\gamma} } [\nabla f]_\gamma |u-v|^{1+\gamma}.
  \end{equation*}
\end{lemma}

Now we are in a position to study the homogenization of SDEs with multiplicative stable noise.

\begin{proposition}\label{HomoSDE}
  In the sense of weak convergence on the space $\D$, we have that,
  \begin{equation*}
    X^{x,\epsilon}\ \Rightarrow X^x, \quad \text{where } X^x_t:=x+\bar Ct+\bar L_t,
  \end{equation*}
  as $\e\to0$. The homogenized coefficient $\bar C$ is given by
  \begin{equation}\label{HomoCoef1}
    \bar C=\int_{\T^d}(I+\nabla \hat b)c(x)\mu(dx),
  \end{equation}
  and $\{\bar L_t\}_{t\ge0}$ is a symmetric $\alpha$-stable L\'evy processes with jump intensity measure
  \begin{equation}\label{homo_nu}
    \bar\nu(A)=\int_{\Ro d}\int_{\T^d}\ind_A(\sigma(x,y))\mu(dx)\nu^\alpha(dy), \quad A\in\B(\Ro d).
  \end{equation}
\end{proposition}

\begin{proof}
  Since $\hat b$ is bounded, the theorem will follow if we prove that $\hat X^{x,\epsilon}\Rightarrow X^x$, as $\epsilon\to 0$. 
  By applying It\^o's formula, and note that $\hat b\in\C^{\alpha+\beta}(\T^d)$ is the solution of Poisson equation \eqref{b_hat},
  \begin{equation*}
    \begin{split}
      \hat X_t^{x,\epsilon} = &\ x+\int_0^t(I+\nabla \hat b) \textstyle{ c\left(\frac{X_{s}^{x,\epsilon}}{\epsilon}\right) } ds - \int_0^t \frac{1}{\epsilon^{\alpha-1}} \A^{\bar\sigma,\nu^\alpha} \hat b \left(\frac{X_{s}^{x,\epsilon}}{\epsilon}\right)ds +\int_0^t \epsilon \A^{\sigma_\e,\nu^\alpha} \hat b_\e\left(X_{s}^{x,\e}\right) ds \\
        & +\int_0^t\int_{\Ro d}\epsilon\left[\hat b_\epsilon\left(X_{s-}^{x,\epsilon}+ \sigma_\epsilon \left(X_{s-}^{x,\epsilon},y\right) \right)-\hat b_\epsilon\left(X_{s-}^{x,\epsilon}\right) \right] \tilde N^\alpha(dy,ds) \\
        & +\int_0^t\int_{\Bo} \sigma_\epsilon\left(X_{s-}^{x,\epsilon},y \right)\tilde N^\alpha(dy,ds) +\int_0^t\int_{B^c} \sigma_\epsilon\left(X_{s-}^{x,\epsilon},y \right) N^\alpha(dy,ds) \\
        =: &\ x+ \Lambda_1^\e(c)_t- \Lambda_2^\e(\hat b,\A^{\bar\sigma,\nu^\alpha})_t +\Lambda_3^\e(\hat b,\A^{\sigma_\e,\nu^\alpha})_t+ \Lambda_4^\e(\hat b, \tilde N^\alpha)_t +\Lambda_5^\e(\sigma, \tilde N^\alpha)_t+\Lambda_6^\e(\sigma, N^\alpha)_t.
    \end{split}
  \end{equation*}
  where $\hat b_\epsilon(x):=\hat b\left(\frac{x}{\e}\right)$, $\sigma_\epsilon(x,y):=\sigma\left(\frac{x}{\e},y\right)$.

  For the last three stochastic integral terms, we figure out the characteristics of them as semimartingales (cf. \cite[Proposition IX.5.3]{JS13}). Choose the truncation function $h_1(x)=x\ind_B(x)$. Denote by $\Xi^\e(s,y):=\epsilon[\hat b_\epsilon(X_{s-}^{x,\epsilon}+ \sigma_\epsilon (X_{s-}^{x,\epsilon},y))-\hat b_\epsilon(X_{s-}^{x,\epsilon})]$. Note that $\Xi^\e(\cdot,0)\equiv0$ by virtue of $\sigma(\cdot,0)\equiv0$ as mentioned in Remark \ref{rem_sigma}.(3). Then the characteristics of $\Lambda_4^\e(\hat b, \tilde N^\alpha)$ associated with $h_1$ is given by
  \begin{equation*}\left\{
    \begin{aligned}
      B_4^\e(t) &= -\int_0^t\int_{\Ro d} \Xi^\e(s,y)\ind_{B^c}(\Xi^\e(s,y))\nu^\alpha(dy)ds, \\
      C_4^\e & \equiv 0, \\
      \nu_4^\e(A\times[0,t]) &= \int_0^t\int_{\Ro d} \ind_A(\Xi^\e(s,y))\nu^\alpha(dy)ds, \quad A\in\B(\Ro d).
    \end{aligned} \right.
  \end{equation*}
  The characteristics of $\Lambda_5^\e(\sigma, \tilde N^\alpha)+\Lambda_6^\e(\sigma, N^\alpha)$ is given by
  \begin{equation*}\left\{
    \begin{aligned}
      B_{5+6}^\e(t) &= \int_0^t\int_{\Ro d} \sigma_\epsilon\left(X_{s-}^{x,\epsilon},y \right)\left[\ind_B\left(\sigma_\epsilon\left(X_{s-}^{x,\epsilon},y \right)\right)-\ind_B(y)\right]\nu^\alpha(dy)ds, \\
      C_{5+6}^\e & \equiv 0, \\
      \nu_{5+6}^\e(A\times[0,t]) &= \int_0^t\int_{\Ro d} \ind_A\left(\sigma_\epsilon\left(X_{s-}^{x,\epsilon},y \right)\right)\nu^\alpha(dy)ds, \quad A\in\B(\Ro d).
    \end{aligned} \right.
  \end{equation*}
  By the same argument as in \eqref{trans}, we have $B_{5+6}^\e\equiv0$.

  Then the theorem is a consequence of the functional central limit theorem in \cite[Theorem VIII.2.17]{JS13}, and the following lemma whose proof is technical and provided in Appendix \ref{app}.
\end{proof}

\begin{lemma}\label{FCLT}
  For any $t\in\R_+$, and any bounded continuous function $f:\R^d\to\R$ which vanishes in a neighbourhood of the origin, the following convergences hold in probability $\P$ when $\e\to0$:
  \begin{enumerate}[(i)]
    \item $\sup_{0\le s\le t} \left| \Lambda_1^\e(c)_s- \bar Cs \right| \to 0$;
    \item $\sup_{0\le s\le t} \left | \Lambda_3^\e(\hat b,\A^{\sigma_\e,\nu^\alpha})_s-\Lambda_2^\e(\hat b,\A^{\bar\sigma,\nu^\alpha})_s \right| \to 0$;
    \item $\sup_{0\le s\le t} |B_4^\e(s)| \to 0$;
    \item $\int_0^t\int_{\Ro d}f(x)\nu_4^\e(dx,ds) \to 0$;
    \item $\int_0^t\int_{\Ro d}f(x)\nu_{5+6}^\e(dx,ds) \to t\int_{\Ro d}f(x)\bar\nu(dx)$;
  \end{enumerate}
  where $\bar C$ and $\bar\nu$ are defined in \eqref{HomoCoef1} and \eqref{homo_nu}, respectively.
\end{lemma}

\subsection{Homogenization of linear nonlocal PDEs}

Define
\begin{equation*}
  Y_t^\epsilon := \int_0^t \textstyle{ \left(\frac{1}{\epsilon^{\alpha-1}} e\left(\frac{X_s^{x,\epsilon}}{\epsilon}\right)+ g\left(\frac{X_s^{x,\epsilon}}{\epsilon}\right)\right) } ds.
\end{equation*}
Thanks to Proposition \ref{wellposed-parabolic}, the nonlocal PDE \eqref{ue} has a unique mild solution, which is given by the Feynman-Kac formula,
\begin{equation*}
  u^\epsilon(t,x)= \E\left[u_0(X_t^{x,\epsilon})\exp(Y_t^\epsilon)\right].
\end{equation*}

Similar to $\hat X^{x,\e}$, we define
\begin{equation*}
  \hat Y_t^\epsilon:=Y_t^\epsilon+\epsilon \textstyle{ \left(\hat e\left(\frac{Y_t^\epsilon}{\epsilon}\right)-\hat e\left(\frac{x}{\epsilon}\right)\right) }.
\end{equation*}
Here $\hat e\in\C_\mu^{\alpha+\beta}(\T^d)$, thanks to Proposition \ref{wellposed-Poisson} and Assumption \ref{center}, is the unique solution of the Poisson equation
\begin{equation*}
  \L^\alpha \hat e + e = 0,
\end{equation*}
with $\L^\alpha$ given by \eqref{L}. 
Again using It\^o's formula,
\begin{equation*}
  \begin{split}
    \hat Y_t^\epsilon =&\ \int_0^t(g+\nabla \hat ec) {\textstyle{ \left(\frac{X_{s}^{x,\epsilon}}{\epsilon}\right) }} ds - \int_0^t {\textstyle{ \frac{1}{\epsilon^{\alpha-1}} \A^{\bar\sigma,\nu^\alpha} \hat e \left(\frac{X_{s}^{x,\epsilon}}{\epsilon}\right) }} ds +\int_0^t \epsilon \A^{\sigma_\e,\nu^\alpha} \hat e_\e\left(X_{s}^{x,\e}\right) ds \\
        & +\int_0^t\int_{\Ro{d}}\epsilon\left[\hat e_\epsilon\left(X_{s-}^{x,\epsilon}+ \sigma_\epsilon \left(X_{s-}^{x,\epsilon},y\right) \right)-\hat e_\epsilon\left(X_{s-}^{x,\epsilon}\right) \right] \tilde N^\alpha(dy,ds) \\
        =:&\ \Lambda_1^\e(c,g)_t- \Lambda_2^\e(\hat e,\A^{\bar\sigma,\nu^\alpha})_t +\Lambda_3^\e(\hat e,\A^{\sigma_\e,\nu^\alpha})_t+ \Lambda_4^\e(\hat e, \tilde N^\alpha)_t.
  \end{split}
\end{equation*}

Then in the same way as the proof of Proposition \ref{HomoSDE}, we have the convergence of $Y^\e$ to a deterministic path in distribution, as well as in probability, since the limit is deterministic (see, e.g., \cite[Theorem 2.7.(iii)]{Van98}).
\begin{lemma}
  In the sense of weak convergence on the space $\D$, both $Y^\e$ and $\hat Y^\e$ converge in distribution, and hence in probability, to the deterministic path $y(t):=\bar Et$ as $\e\to0$, where the homogenized coefficient $\bar E$ is given by
  \begin{equation}\label{HomoCoef2}
    \bar E:=\int_{\T^d}(g+\nabla \hat ec)(x)\mu(dx).
  \end{equation}
\end{lemma}

%

Now we are in the position to prove the main result of this section. Since $\hat b$ and $\hat e$ are bounded on $\R^d$, $u^\e$ has the same limit behavior as
\begin{equation*}
  \hat u^\e(t,x):= \E[u_0(X_t^{x,\epsilon})\exp(\hat Y_t^\epsilon)]
\end{equation*}
when $\e\to 0$. Hence, we only need to show $\hat u^\epsilon(t,x)\to u(t,x), \e\to 0$ for any $t\ge 0,x\in\R^d$, where $u$ has the following  Feynman-Kac representation by virtue of Proposition \ref{wellposed-parabolic},
\begin{equation*}
  u(t,x)=\E[u_0(X^x_t)]e^{y(t)}.
\end{equation*}

\begin{proof}[\textbf{Proof of Theorem \ref{HomoPDE}}]
  For the convenience of notation, we shall write $\Lambda_1^\e(c,g)_t$, $\Lambda_2^\e(\hat e,\A^{\sigma,\nu^\alpha})_t$, $\Lambda_3^\e(\hat e,\A^{\sigma_\e,\nu^\alpha})_t$, $\Lambda_4^\e(\hat e, \tilde N^\alpha)_t$ as $\Lambda_1^\e(t)$, $\Lambda_2^\e(t)$, $\Lambda_3^\e(t)$, $\Lambda_4^\e(t)$, respectively. We fix a $t\in\R_+$.

  Firstly, we prove the uniform integrability of the set $\{e^{\Lambda_4^\e(t)}|0<\e\le1\}$ for each $t\in\R_+$. This follows by proving that it is uniformly bounded in $L^2(\Omega,\P)$. Denoting the integrand in $\Lambda_4^\e(\hat e, \tilde N^\alpha)$ by
  \begin{equation*}
    \Gamma^\e(s,y):=\epsilon\left[\hat e_\epsilon\left(X_{s-}^{x,\epsilon}+ \sigma_\epsilon \left(X_{s-}^{x,\epsilon},y\right) \right)-\hat e_\epsilon\left(X_{s-}^{x,\epsilon}\right) \right].
  \end{equation*}
  Then by It\^o's formula,
  \begin{equation*}
    \begin{split}
       e^{2\Lambda_4^\e(t)} =&\ 1-\int_0^t\int_{B^c} 2e^{2\Lambda_4^\e(s-)}\Gamma^\e(s,y) \nu^\alpha(dy)ds \\
         &\ +\int_0^t\int_{\Bo} e^{2\Lambda_4^\e(s-)}\left(e^{2\Gamma^\e(s,y)}-1\right)\tilde N^\alpha(dy,ds) \\
         &\ +\int_0^t\int_{B^c} e^{2\Lambda_4^\e(s-)} \left(e^{2\Gamma^\e(s,y)}-1\right) N^\alpha(dy,ds) \\
         &\ +\int_0^t\int_{\Bo} e^{2\Lambda_4^\e(s-)} \left[e^{2\Gamma^\e(s,y)}-1- 2\Gamma^\e(s,y)\right] \nu^\alpha(dy)ds.
    \end{split}
  \end{equation*}
  Since $\hat e$ is bounded, $\Gamma^\e$ has a uniform bound for all $\e>0$. Then there exists a large constant $C>0$ such that for each $\e>0$ and $t\in\R_+$,
  \begin{equation*}
    \E\int_0^t\int_{B^c}\left|e^{2\Lambda_4^\e(s-)} \left(e^{2\Gamma^\e(s,y)}-1-2\Gamma^\e(s,y)\right)\right|^2 \nu^\alpha(dy)ds \le C\nu^\alpha(B^c)\E\int_0^t e^{2\Lambda_4^\e(s-)} ds<\infty.
  \end{equation*}
  Hence combining these, there exists $\theta\in(0,1)$ such that
  \begin{equation*}
    \begin{split}
       \E e^{2\Lambda_4^\e(t)} & =1+\E\int_0^t\int_{\Bo} e^{2\Lambda_4^\e(s-)} \left[e^{2\Gamma^\e(s,y)}-1- 2\Gamma^\e(s,y)\right] \nu^\alpha(dy)ds \\
         &\le 1+\E\int_0^t e^{2\Lambda_4^\e(s-)} \int_{\Bo}2e^{2\theta\Gamma^\e(s,y)} |\Gamma^\e(s,y)|^2 \nu^\alpha(dy)ds \\
         &\le 1+C(\|\hat e\|_0)\E\int_0^t e^{2\Lambda_4^\e(s-)} \int_{\Bo} |\Gamma^\e(s,y)|^2 \nu^\alpha(dy)ds.
    \end{split}
  \end{equation*}
  As shown in the proof of part (iii) and (iv) in Lemma \ref{FCLT},
  \begin{equation*}
    \int_{\Bo} |\Gamma^\e(s,y)|^2 \nu^\alpha(dy) \le \e^2\|\hat e\|_1^2 \int_{\Bo} |\sigma_\e(X_{s-}^{x,\e},y)|^2 \nu^\alpha(dy) \le \frac{|\S^{d-1}|}{2-\alpha} \e^2\|\hat e\|_1^2 \|\psi\|_{L^\infty}^2.
  \end{equation*}
  Thus,
  \begin{equation*}
    \E e^{2\Lambda_4^\e(t)} \le 1+\e^2 C(\alpha,|\S^{d-1}|,\|\hat e\|_1,\|\psi\|_{L^\infty})\int_0^t\E e^{2\Lambda_4^\e(s-)}ds.
  \end{equation*}
  By Gr\"onwall's inequality, the uniform boundness of $\{e^{\Lambda_4^\e(t)}|0<\e\le1\}$ in $L^2(\Omega,\P)$ follows.

  Secondly, the proof of Lemma \ref{FCLT} shows that the set $\{\Lambda_3^\e(t)-\Lambda_2^\e(t)| 0<\e\le1\}$ is bounded.
  The set $\{\Lambda_1^\e(t)| 0<\e\le1\}$ is bounded by virtue of the boundness of $c,g$ and $\hat e$. Also since $u_0$ is periodic and continuous, $\{u_0(X_t^{x,\e})|0<\e\le1\}$ is bounded. Thus, the set $\{u_0(X_t^{x,\epsilon})\exp(\hat Y_t^\epsilon)|0<\e\le1\}$ is uniformly integrable.

  Finally, we pass to the limit. It is easy to see that $e^{\hat Y_t^\e}\to e^{y(t)}$ in probability as $\e\to0$, by the continuous mapping theorem (see, for instance, \cite[Theorem 2.3.(ii)]{Van98}).
  Then for any subsequence $\{\e_n\}\to0$, there exists a subsubsequence $\{\e_{n_k}\}\to0$ such that $e^{\hat Y_t^{\e_{n_k}}}\to e^{y(t)}$ almost uniformly (cf. \cite[Lemma 4.2]{Kal06}). That is, for any $\rho>0$, there exists a set $N\in\F$ with $\P(N)\le\rho$, such that
  \begin{equation}\label{a.un.}
    \left\|e^{\hat Y_t^{\e_{n_k}}}-e^{y(t)}\right\|_{L^\infty(N^c,\P)}\to 0, \quad k\to\infty.
  \end{equation}
  By the boundness of $u_0$, we know the set $\{u_0(X_t^{x,\epsilon})[\exp(\hat Y_t^\epsilon)-\exp(y(t))]|0<\e\le1\}$ is also uniformly integrable. Then for any $\delta>0$, there exist $\rho_0>0$ and $N_0\in\F$ with $\P(N_0)\le\rho_0$, such that
  \begin{equation}\label{uni_abs_cont}
    \E\left|u_0(X_t^{x,\e})\left(e^{\hat Y_t^\e}- e^{y(t)}\right)\ind_{N_0}\right|<\delta.
  \end{equation}
  Now along the sequence $\{\e_{n_k}\}$, we combining \eqref{a.un.} with \eqref{uni_abs_cont} to get
  \begin{equation*}
    \begin{split}
       \E\left|u_0(X_t^{x,\e_{n_k}})\left(e^{\hat Y_t^{\e_{n_k}}}- e^{y(t)}\right)\right| &\le \E\left|\cdots\ind_{N_0}\right| + \E\left|\cdots\ind_{N_0^c}\right| \\
         &\le \delta+\|u_0\|_{L^\infty}\P(N_0^c)\left\|e^{\hat Y_t^{\e_{n_k}}}-e^{y(t)}\right\|_{L^\infty(N^c,\P)} \\
         &\le 2\delta.
    \end{split}
  \end{equation*}
  To summarize these together, for any subsequence $\{\e_n\}\to0$, there exists a subsubsequence $\{\e_{n_k}\}\to0$ such that
  \begin{equation*}
    \E\left|u_0(X_t^{x,\e_{n_k}})\left(e^{\hat Y_t^{\e_{n_k}}}- e^{y(t)}\right)\right|\to 0, \quad k\to \infty,
  \end{equation*}
  which implies that the convergence holds on the whole line $0<\e\le1$. On the other hand, by Proposition \ref{HomoSDE}, we know that $\E[u_0(X_t^{x,\e})] \to \E[u_0(X_t^x)]$ as $\e\to0$. The result \eqref{u_conv} follows from
  \begin{equation*}
    \begin{split}
      |\hat u^\e(t,x) - u(t,x)| &= \left| \E\left[ u_0(X_t^{x,\e})e^{\hat Y_t^\e} \right] - \E[u_0(X^x_t)]e^{y(t)} \right| \\
      &\le \E\left|u_0(X_t^{x,\e})\left(e^{\hat Y_t^{\e}}- e^{y(t)}\right)\right| + \left| \E\left[ u_0(X_t^{x,\e}) \right] - \E\left[ u_0(X_t^x)\right] \right| e^{y(t)}.
    \end{split}
  \end{equation*}
\end{proof}

\begin{remark}
  We close this section by some comments for the proof of Theorem \ref{HomoPDE}. In \cite{Par99}, the author applied Girsanov's transform to get rid of the stochastic integral term involved in $\hat Y^\e$, since this term may not possess the uniformly integrability. While in our case, since the stochastic integral term in $Y_t^\e$ has an infinitesimal integrand $\Gamma^\e(s,y)$, the uniform integrability of $\{\exp(\hat Y_t^\e)|0<\e\le1\}$ is easier to treat.


\end{remark}

\appendix

\section{Proofs of auxiliary results}\label{app}

%

In this part, we provide missing proofs of some auxiliary results in our paper.

\begin{proof}[\textbf{Proof of Lemma \ref{MP}}]
  Note that the nonlocal operator $\A^{\bar\sigma,\nu^\alpha}$ can be rewritten as the form \eqref{A_rewrite}. For $u\in\C_b^{1+\gamma}(\R^d)$, we have
  \begin{equation}\label{difference-estimate-1}
    |u(x+z)-u(x)- z\cdot \nabla u(x)| \le |z| \int_0^1 |\nabla u(x+rz)-\nabla u(x)|dr \le \frac{[\nabla u]_\gamma}{1+\gamma}|z|^{1+\gamma}.
  \end{equation}
  Then by \eqref{Levy-kernel}, there exists a constant $C>0$ such that
  \begin{equation*}
    \begin{split}
       |\L^\alpha u(x)|\le &\ \int_{\Bo} |u(x+z)-u(x)- z\cdot \nabla u(x)| \nu^{\bar\sigma,\alpha}(x,dz) \\
         &\ + \int_{B^c} | u(x+z)-u(x)| \nu^{\bar\sigma,\alpha}(x,dz) + |b(x)\cdot\nabla u(x)| \\
         \le&\ 2\|u\|_{1+\gamma}\left( \int_{\Ro d} (|z|^{1+\gamma}\wedge1) \nu^{\bar\sigma,\alpha}(x,dz)+\|b\|_0 \right) \\
         \le&\ C\|u\|_{1+\gamma}.
    \end{split}
  \end{equation*}
  Based on this estimate, the rest of the proof is exactly the same as that of \cite[Proposition 3.2]{Pri12}, even though it is set up with $\bar\sigma(\cdot,y)\equiv y$ there.
\end{proof}

\begin{proof}[\textbf{Proof of Lemma \ref{Feller}}]
  Using \eqref{difference-estimate-1} with $\gamma=\alpha+\beta-1$, one can find that for $u\in\C^{\alpha+\beta}(\T^d)$,
  \begin{equation*}
    \big|u( x+\bar\sigma(x,y))-u(x)- \bar\sigma(x,y)\cdot \nabla u(x)\big| \le \textstyle{\frac{1}{\alpha+\beta}} [\nabla u]_{\alpha+\beta-1} |\bar\sigma(x,y)|^{\alpha+\beta}.
  \end{equation*}
  Combining this with \eqref{moment}, a straightforward application of the dominated convergence theorem yields that $\lim_{y\to x}\L^\alpha u(y)=\L^\alpha u(x)$ for any $u\in\C^{\alpha+\beta}(\T^d)$ and $x\in\T^d$. This amounts to saying that $\L^\alpha(\C^{\alpha+\beta}(\T^d))\subset\C(\T^d)$. Therefore, the operator
  \begin{equation*}
    \L^\alpha:\C(\T^d)\supset\C^{\alpha+\beta}(\T^d)\to\C(\T^d)
  \end{equation*}
  is a densely defined unbounded operator on $\C(\T^d)$.

  Now Lemma \ref{MP} implies that for any $\kappa>0$ and $u\in\C^{\alpha+\beta}(\T^d)$, $\|(\kappa-\L^\alpha)u\|_0\ge\kappa\|u\|_0$, that is, $\L^\alpha$ is dissipative. By Proposition \ref{res_prob}, we have $\C^{\beta}(\T^d)\subset (\kappa-\L^\alpha)(\C^{\alpha+\beta}(\T^d))$ for any $\kappa>0$, which yields that the operator $\kappa-\L^\alpha$ has dense range in $\C(\T^d)$. In addition, $\L^\alpha$ satisfies the positive maximum principle, due to the equivalent form \eqref{A_rewrite} of $\A^{\bar\sigma,\nu^\alpha}$ and Courr\`ege's theorem (see \cite[Corollary 4.5.14]{Jac01}). Now the final assertion follows form the celebrate Hille-Yosida-Ray Theorem (see, for instance, \cite[Theorem 4.2.2]{EK09}).
\end{proof}

\begin{proof}[\textbf{Proof of Lemma \ref{martingale-prob}}]
  The existence of solution of the martingale problem is in \cite[Proposition 3]{MP14}. Taking Lemma \ref{Feller} into account, the uniqueness and the Feller property follow from \cite[Theorem 4.4.1]{EK09}. The existence of transition density and the two estimates can be found in \cite[Theorem 1.5]{CZ18}. 
\end{proof}

\begin{proof}[\textbf{Proof of Lemma \ref{Phi}}]
  By the estimate in Corollary \ref{no-large-jump}, we have
  \begin{equation*}
    \kappa^{\frac{\alpha+\beta-1}{\alpha}} \|\nabla\hat b_\kappa \|_0 \le C \|b\|_\beta, \quad \kappa>\kappa_*.
  \end{equation*}
  Now by choosing $\kappa>\kappa_*\vee(2C \|b\|_\beta)^{\frac{\alpha}{\alpha+\beta-1}}$, we get that $\|\nabla\hat b_\kappa\|_0\le\frac{1}{2}$. Thus
  \begin{equation*}
    \frac{1}{2}|x_1-x_2|\le \big|\Phi_\kappa(x_1)-\Phi_\kappa(x_2)\big| \le\frac{3}{2}|x_1-x_2|,
  \end{equation*}
  i.e., $\Phi_\kappa$ is bi-Lipschitz. In particular, $\Phi_\kappa$ is a $\C^1$-diffeomorphism. Moreover,
  \begin{equation*}
    \nabla(\Phi_\kappa^{-1})= \text{Inv}\circ\nabla\Phi_\kappa\circ\Phi_\kappa^{-1},
  \end{equation*}
  where the matrix inverse map $\text{Inv}:\text{GL}(\R^d) \to\text{GL}(\R^d)$ is of class $\C^\infty$. Note that $\nabla\Phi_\kappa$ is of class $\C^{\alpha+\beta-1}$, $\Phi_\kappa^{-1}$ is of class $\C^1$. It is easy to see that $\nabla(\Phi_\kappa^{-1})$ is of class $\C^{\alpha+\beta-1}$. The second conclusion of the lemma follows. 
\end{proof}

\begin{proof}[\textbf{Proof of Corollary \ref{strong-Markov}}]
  By applying It\^o's formula, it is easy to see that for any $f\in D(\L^\alpha)=\C^{\alpha+\beta}(\T^d)$, the following process is a $(\Omega,\F,\P,\{\F_t\}_{t\ge0})$-martingale
  \begin{equation*}
    \tilde M^f(t):=f(\tilde X^x_t)-f(\tilde X^x_0)-\int_0^t \L^\alpha f(\tilde X^x_s) ds.
  \end{equation*}
  It is easy to see that $\tilde X^x$ has c\`adl\`ag paths almost surely. Let $\P_{\tilde X^x}:=\P\circ\tilde X^x$ be the pushforward probability measure of $\tilde X^x$ on $(\D,\B(\D))$, then $\P_{\tilde X^x}$ is a solution of martingale problem for $(\L^\alpha,\delta_x)$. By the uniqueness of solutions to the martingale problem obtained in Lemma \ref{martingale-prob}, we find that $\P_{\tilde X^x}=\P^x$, the Feller property follows. The strong Markov property follows from \cite[Theorem III.3.1]{RY13}.
\end{proof}

\begin{proof}[\textbf{Proof of Lemma \ref{restriction}}]
  Since $\mu$ is invariant with respect to $\{P_t\}_{t\ge0}$, for any $f\in\C_\mu(\T^d)$ and $t\ge0$, we have
  \begin{equation*}
    \int_{\T^d} P_t f(x) \mu(dx)=\int_{\T^d} f(x) \mu(dx)=0.
  \end{equation*}
  That is, $\C_\mu(\T^d)$ is $\{P_t\}_{t\ge0}$-invariant, in the sense that $P_t(\C_\mu(\T^d))\subset\C_\mu(\T^d)$ for all $t\ge0$. The lemma then follows from the corollary in \cite[Subsection II.2.3]{EN00}.
\end{proof}

\begin{proof}[\textbf{Proof of Lemma \ref{res_prob-center}}]
  Since $f$ is centered with respect to $\mu$, by Lemma \ref{inv} we have
  \begin{equation}\label{SG_est}
  \|P_t f\|_0 \le C\|f\|_0e^{-\rho t}.
  \end{equation}
  Note the fact that $\mu$ is invariant with respect to $\{P_t\}_{t\ge0}$. Then combining \eqref{SG_est} and the representation \eqref{rep_u_lambda}, a straightforward application of Fubini's theorem implies that
  \begin{equation*}
    \begin{split}
       \int_{\T^d} u_\kappa(x) \mu(dx) & = \int_{\T^d} \int_0^\infty e^{-\kappa t} P_t f(x) dt \mu(dx) = \int_0^\infty e^{-\kappa t} \left(\int_{\T^d} P_t f(x) \mu(dx)\right) dt \\
         & = \int_0^\infty e^{-\kappa t} \left(\int_{\T^d} f(x) \mu(dx)\right) dt =0.
    \end{split}
  \end{equation*}
  That is, $u_\kappa$ is also centered with respect to $\mu$.
\end{proof}

\begin{proof}[\textbf{Proof of Lemma \ref{FCLT}}]
  (i). By Proposition \ref{ergodic_thm}, the convergence in probability of the first integral is immediate,
  \begin{equation*}
    \sup_{0\le s\le t} \left| \Lambda_1^\e(c)_s- \bar Cs \right| \le \int_0^t \left|(I+\nabla \hat b) c \textstyle{ \left(\frac{X_{s}^{x,\epsilon}}{\epsilon}\right) } -\bar C\right|ds \to 0, \quad \e\to 0.
  \end{equation*}

  (ii). Note that $\nu^\alpha(\epsilon A)=\epsilon^{-\alpha}\nu^\alpha(A)$, $A\in\B(\Ro d)$. Then a change of variable yields
  \begin{equation*}
    \begin{split}
      &\ \A^{\sigma_\e,\nu^\alpha} \hat b_\e(x) \\
      =&\ \frac{1}{\e^{\alpha-1}} \int_{\Ro d} \textstyle{ \left[ \hat b\left( \DivEps x + \DivEps 1\sigma\left(\DivEps x, \e y\right) \right) - \hat b\left( \DivEps x \right) - \DivEps 1 \sigma^i\left(\DivEps x, \e y\right) \partial_i\hat b\left(\DivEps x\right) \ind_B(\e y) \right] } \nu^\alpha(dy).
    \end{split}
  \end{equation*}
  By the oddness condition in Assumption \ref{sigma},
  for any $\delta>0$, we have
  \begin{equation}\label{diff-A}
    \begin{split}
       &\ \left|\epsilon \A^{\sigma_\e,\nu^\alpha} \hat b_\e(x)- \textstyle{ \frac{1}{\epsilon^{\alpha-1}} } \A^{\bar\sigma,\nu^\alpha} \hat b \textstyle{ \left(\DivEps x\right) } \right| \\
         \le&\ \frac{1}{\e^{\alpha-1}} \int_{B_\delta\setminus\{0\}} \big| \big[ \hat b \textstyle{ \left( \DivEps x + \DivEps 1\sigma\left(\DivEps x, \e y\right) \right) } - \hat b \textstyle{ \left( \DivEps x + \bar\sigma\left(\DivEps x, y\right) \right) } \\
         &\qquad\qquad\ - \textstyle{ \left(\DivEps 1 \sigma^i\left(\DivEps x, \e y\right) - \bar\sigma^i\left(\DivEps x, y\right) \right) } \partial_i\hat b\left(\DivEps x\right) \big] \big| \nu^\alpha(dy) \\
         &\ +\frac{1}{\e^{\alpha-1}} \int_{B_\delta^c} \textstyle{ \left| \left[ \hat b\left( \DivEps x + \DivEps 1\sigma\left(\DivEps x, \e y\right) \right) - \hat b\left( \DivEps x + \bar\sigma\left(\DivEps x, y\right) \right)  \right]\right| } \nu^\alpha(dy) \\
         =:&\ I_1^\e(x) + I_2^\e(x),
    \end{split}
  \end{equation}
  where $B_\delta$ is the open ball in $\R^d$ centering at the origin with radius $\delta$, $I_1^\e$ and $I_2^\e$ are the two integral terms in the second equality. We let $\delta = \e^\gamma$ for some $\gamma\in\R$ that will be chosen latter. It follows from Lemma \ref{difference-estimate-2} that
  \begin{equation*}
    \begin{split}
       I_1^\e(x) & \le \frac{\|\hat b\|_{\alpha+\beta}}{\alpha+\beta} \frac{1}{\e^{\alpha-1}} \int_{B_\delta\setminus\{0\}} \textstyle{ \left| \DivEps 1 \sigma\left(\DivEps x, \e y\right) - \bar\sigma\left(\DivEps x, y\right) \right|}^{\alpha+\beta} \nu^\alpha(dy).
    \end{split}
  \end{equation*}
  Recall that $\sigma(x,0) = 0$ and $\bar\sigma(x,y) = \nabla_y\sigma(x,0)y$ as in \eqref{sigma-bar} and \eqref{sigma0-bar}. Then by \eqref{difference-estimate-1},
  \begin{equation*}
    \begin{split}
      \textstyle{ \left| \DivEps 1 \sigma\left(\DivEps x, \e y\right) - \bar\sigma\left(\DivEps x, y\right) \right| } &= \textstyle{ \DivEps 1 \left| \sigma\left(\DivEps x, \e y\right) - \sigma\left(\DivEps x, 0\right) - \e \nabla_y\sigma\left(\DivEps x,0\right) y \right| } \le \textstyle{ \frac{\e^\lambda [\nabla_y \sigma\left(\DivEps x,\cdot\right)]_\lambda}{1+\lambda} } |y|^{1+\lambda}.
    \end{split}
  \end{equation*}
  Hence, using the assumption \eqref{sigma-Holder}, we have
  \begin{equation}\label{I1}
    \begin{split}
      I_1^\e(x) &\le \frac{\|\hat b\|_{\alpha+\beta} \sup_{x\in\T^d}[\nabla_y \sigma\left(x, \cdot\right)]_\lambda^{\alpha+\beta} } {(\alpha+\beta)(1+\lambda)^{\alpha+\beta}} \e^{\lambda(\alpha+\beta)-\alpha+1} \int_{B_\delta\setminus\{0\}} |y|^{(1+\lambda)(\alpha+\beta)} \nu^\alpha(dy) \\
      &\le C\textstyle{ \left(d, \alpha, \beta, \lambda, \|\hat b\|_{\alpha+\beta}, \sup_{x\in\T^d}[\nabla_y \sigma\left(x, \cdot\right)]_\lambda \right) } \e^{\lambda(\alpha+\beta)-\alpha+1+ \gamma[(1+\lambda)(\alpha+\beta)-\alpha] }.
    \end{split}
  \end{equation}
  On the other hand, using a change of variable once again,
  \begin{equation}\label{I2}
    \begin{split}
      I_2^\e(x) &= \e \int_{B_{\e\delta}^c} \textstyle{ \left| \left[ \hat b\left( \DivEps x + \DivEps 1\sigma\left(\DivEps x, y\right) \right) - \hat b\left( \DivEps x + \bar\sigma\left(\DivEps x, \DivEps y\right) \right)  \right]\right| } \nu^\alpha(dy) \le 2 \|\hat b\|_0 \e \nu^\alpha(B_{\e\delta}^c) \\
      &\le C(d, \|\hat b\|_0) \e^{1-\alpha(1+\gamma)}.
    \end{split}
  \end{equation}
  Now we select $\gamma$ satisfying
  \begin{equation}\label{gamma}
    -\frac{\lambda(\alpha+\beta)-\alpha+1}{(1+\lambda)(\alpha+\beta)-\alpha} < \gamma < \frac{1}{\alpha} - 1,
  \end{equation}
  One can easily verify that the interval in \eqref{gamma} is not null, since $\lambda>\alpha-1$ as assumed in the regularity condition in \ref{sigma}. Combining \eqref{diff-A}-\eqref{gamma}, we conclude that
  \begin{equation*}
    \sup_{x\in\T^d} \left|\epsilon \A^{\sigma_\e,\nu^\alpha} \hat b_\e(x)- \textstyle{ \frac{1}{\epsilon^{\alpha-1}} } \A^{\bar\sigma,\nu^\alpha} \hat b \textstyle{ \left(\DivEps x\right) } \right| \to 0, \quad \text{as } \e\to0.
  \end{equation*}
  Therefore,
  \begin{equation*}
    \begin{split}
       \sup_{0\le s\le t} \left | \Lambda_3^\e(\hat b,\A^{\sigma_\e,\nu^\alpha})_s-\Lambda_2^\e(\hat b,\A^{\bar\sigma,\nu^\alpha})_s \right| \le&\ \int_0^t \textstyle{ \left|\epsilon \A^{\sigma_\e,\nu^\alpha} \hat b_\e\left(X_{s}^{x,\e}\right)- \frac{1}{\epsilon^{\alpha-1}} \A^{\bar\sigma,\nu^\alpha} \hat b \left(\DivEps{X_{s}^{x,\e}} \right)\right| } ds \\
         \le&\ t \sup_{x\in\T^d} \left|\epsilon \A^{\sigma_\e,\nu^\alpha} \hat b_\e(x)- \textstyle{ \frac{1}{\epsilon^{\alpha-1}} \A^{\bar\sigma,\nu^\alpha} \hat b \left(\DivEps x\right) } \right| \\
         \to&\ 0, \quad \e\to 0.
    \end{split}
  \end{equation*}

  (iii) and (iv). Since $\Xi^\e$ is integrable with respect to $\tilde N^\alpha$,
  the third characteristic of $\Lambda_2^\e$ satisfies that $\int_0^t\int_{\Ro d}(|x|^2\wedge1)\nu_2^\e(dx,ds)<\infty$ for each $\e>0$ and $t\in\R_+$ (cf. \cite[Proposition II.2.9]{JS13}). By the hypothesis, there exist $\rho>0$ and $M>0$ such that $|f|\le M$ on $B_\rho^c$ and $f=0$ on $B_\rho$. Then for any $t\in\R_+$,
  \begin{equation*}
    \int_0^t\int_{\Ro d}f(x)\nu_2^\e(dx,ds) \le M\int_0^t\int_{\Ro d}\ind_{B_\rho^c}(x) \nu_2^\e(dx,ds),
  \end{equation*}
  which goes to zero almost surely as $\e\to 0$ by the boundness of $\hat b$ and the dominated convergence theorem, and (iv) follows.

  For $B_2^\e$, we have the estimate
  \begin{equation*}
    \begin{split}
       \sup_{0\le s\le t} |B_2^\e(s)| & \le \left[\int_0^t\int_{\Ro d}|x|^2\nu_2^\e(dx,ds) \right]^{\frac{1}{2}} \left[\int_0^t\int_{\Ro d} \ind_{B^c}(x) \nu_2^\e(dx,ds)\right]^{\frac{1}{2}} \\
         & =: \sqrt{J^\e_1}\cdot \sqrt{J^\e_2}.
    \end{split}
  \end{equation*}
  By (iv) and a usual approximation procedure, $J^\e_2$ goes to zero surely as $\e\to 0$. For $J^\e_1$,
  \begin{equation*}
    \begin{split}
       J^\e_1 &= \int_0^t\int_{\Ro d} \left|\epsilon\left[\hat b_\epsilon\left(X_{s-}^{x,\epsilon}+ \sigma_\epsilon \left(X_{s-}^{x,\epsilon},y\right) \right)-\hat b_\epsilon\left(X_{s-}^{x,\epsilon}\right) \right] \right|^2 \nu^\alpha(dy) ds \\
         &= \int_0^t\left( \int_{B_\e^c}+\int_{B_\e\setminus\{0\}} \right) |\cdots|^2 \nu^\alpha(dy) ds \\
         &\le \frac{4t\|\hat b\|_0^2|\S^{d-1}|}{\alpha} \e^{2-\alpha} + \|\hat b\|_1^2\int_{B_\e\setminus\{0\}}\int_0^t \textstyle{ \left|\sigma\left( \DivEps{X_{s-}^{x,\e}},y \right)\right|}^2 ds\nu^\alpha(dy).
    \end{split}
  \end{equation*}
  By the growth condition in Assumption \ref{sigma},
  \begin{equation*}
    \int_{B_\e\setminus\{0\}}\int_0^t \textstyle{ \left|\sigma\left( \DivEps{X_{s-}^{x,\e}},y \right)\right|}^2 ds\nu^\alpha(dy)
    \le \frac{t|\S^{d-1}| \e^{2-\alpha}}{2-\alpha} \int_0^t \textstyle{ \left|\phi\left( \DivEps{X_{s-}^{x,\e}}\right)\right|}^2 ds.
  \end{equation*}
  Then (iii) follows from these estimates and Proposition \ref{ergodic_thm}.

  (v). It follows from Proposition \ref{ergodic_thm} that,
  \begin{equation*}
    \begin{split}
       \int_0^t\int_{\Ro d}f(y)\nu^{3+4}_\e(dy,ds) =&\ \int_{\Ro d}\int_0^t f\textstyle{ \left(\sigma\left(\DivEps{X^{x,\e}_{s-}},y\right)\right)} ds\nu^\alpha(dy) \\
         \to &\ t\int_{\Ro d}\int_{\T^d} f(\sigma(x,y))\mu(dx)\nu^\alpha(dy) \\
         =&\ t\int_{\Ro d}f(y)\bar\nu(dy),\quad \e\to 0,
    \end{split}
  \end{equation*}
  where the convergence is in probability.
\end{proof}

\section*{Acknowledgements}
We would like to thank the reviewers for their thoughtful comments and efforts towards improving our manuscript. The research of J. Duan was partly supported by the NSF grant 1620449. The research of Q. Huang was partly supported by China Scholarship Council (CSC), and NSFC grants 11531006 and 11771449. The research of R. Song is supported in part by a grant from the Simons Foundation ($\#$ 429343, Renming Song).

{\footnotesize

}
\end{document}